\documentclass[preprint, 1p, times, 10pt]{elsarticle}

\usepackage{amssymb}
\usepackage{myMathSymbolsAIIPMQP}
\usepackage{amsmath}
\usepackage{lscape}
\usepackage{longtable}
\usepackage{algorithm}
\usepackage{algorithmic}
\usepackage{hyperref}

\newcommand{\resize}[2]{%
\resizebox{#1\hsize}{!}{$ #2 $ }%
}

\makeatletter
\newcommand\footnoteref[1]{\protected@xdef\@thefnmark{\ref{#1}}\@footnotemark}
\makeatother

 \newtheorem{theorem}{Theorem}
 \newtheorem{lemma}[theorem]{Lemma}
 \newdefinition{rmk}{Remark}
 \newproof{proof}{Proof}
 \newtheorem{proposition}{Proposition}
 \newtheorem{corollary}[theorem]{Corollary}

\journal{  }

\begin{document}

\begin{frontmatter}

\title{Active-set prediction in quadratic programming using interior point methods and controlled perturbations{\tnoteref{fnt:disclaimer}}}

\tnotetext[fnt:disclaimer]{
The results of this paper were obtained while the author was  a full-time Ph.D. student at the University of Edinburgh, working under the guidance of Dr. Coralia Cartis. 
This research was supported by the Principal's Career Development Scholarship from the University of Edinburgh.}

\author[Yan]{Yiming Yan}

\ead{yiming.yan@sheffield.ac.uk}

\address[Yan]{
	Department of Automatic Control and System Engineering,
	University of Sheffield, \\
	Sheffield, S1 3JD, 
        United Kingdom. 
}

\begin{abstract}
In this paper, we extend the idea of using controlled perturbations to enhance the capabilities of active-set prediction for interior point methods for convex Quadratic Programming ({\qp}) problems. 
Namely, we consider perturbing the inequality constraints (by a small amount) so as to enlarge the feasible set. We show that if the perturbations are chosen judiciously, then there exists a primal-dual pair of points which is close to the optimal solution of the perturbed problems and the corresponding active and inactive sets at this point are the same as the optimal active and inactive sets at an optimal solution of the original {\qp} problems.
Additionally, we prove that  the optimal tripartition of the original problems can also be predicted by solving the perturbed ones.
Furthermore, encouraging preliminary numerical experience is also presented for the {\qp} case.   
\end{abstract}


\end{frontmatter}

\section{Introduction}
\label{sec:intro}

Consider an inequality-constrained optimisation problem, which minimises (or maximises) the objective function over the feasible region composed of points satisfying the constraints. An \emph{active constraint} is an inequality constraint that holds as equality at a feasible point~\cite{CartisYan2014}. 
Active-set prediction is a technique used to identify the active constraints at an optimal solution of the problem without knowing this solution. Normally it is performed during the solving process of an iterative optimisation algorithm before the final (optimal) iterate is reached, using only information provided by the current iterate or at most several consecutive iterates.

Despite being a class of powerful tools for solving Linear Programming ({\lp}) and Quadratic Programming ({\qp}) problems, 
Interior Point Methods ({\ipms}) are well-known to encounter difficulties with active-set prediction, even for {\lp} problems, due essentially to their constructions~\cite{CartisYan2014}. 
When applied to an inequality constrained optimization problem, {\ipms} generate iterates that belong to the interior of the set determined by the constraints, thus avoiding/ignoring the combinatorial aspect of the solution. This comes at the cost of difficulty in predicting the optimal active constraints that would enable termination, as well as increasing ill-conditioning of the solution process. 

Although active-set prediction techniques for {\ipms} have existed for over a decade,  they suffer from difficulties in making an accurate prediction at the early stage of the iterative process of {\ipms}.  In the case of indicators~\cite{bakry} for example, to get a good prediction, the iterates still need to be close to optimality (small duality gap).
For instance, in~\cite[Table 8.2]{bakry}, at the third from the last iteration, 3 out of the 6 problems predict only a very small portion of the active constraints (less than 15\%) using Tapia indicators. For a review of active-set prediction techniques for {\ipms}, please refer to~\cite{CartisYan2014}.

To address the above mentioned challenge, Cartis and Yan~\cite{CartisYan2014} introduce the idea of using controlled perturbations for {\ipms}  in the purpose of predicting the optimal active set of {\lp} problems.
Namely, in the context of {\lp} problems, they consider perturbing the inequality constraints so as to enlarge the feasible set.
They solve the resulting perturbed problem(s) using a path-following {\ipm} while predicting on the way the active set of the original {\lp} problem; this approach is able to accurately predict the optimal active set of the original problem before the duality gap for the perturbed problem gets too small.
Furthermore, depending on problem conditioning, this prediction can happen sooner than predicting the active set for the perturbed problem or for the original one if no perturbations are used. 

The aim of this paper is to extend this idea to convex {\qp} problems. 
{\qp} problems share many properties of {\lp}, based on which the extension of some results is straightforward (Theorems~\ref{thm:ExistenceOfPerfectPerturbationsForQP} and~\ref{thm:ExistenceOfRelaxedPerturbationsForQP}). 
However, {\qp} problems are not guaranteed to have a strictly complementary solution~\cite{Berkelaar1996,Berkelaar1997}\footnote{
\label{fnt:strictlyCs}
{\ipms} for {\lp} converge to a so-called strictly complementary solution (which always exists for {\lp}~\cite{goldman1956polyhedral}) which leads to a unique optimal active and inactive partition of the constraints~\cite[Section 4.2]{CartisYan2014}. 
Such a solution may not exist for {\qp}.
For the definition the strictly complementary solution, please refer toTheorem 2.3 in~\cite{wright} and the discussion after that.
}
and the existence of a strictly complementary solution is crucial to the theory for the {\lp} case. 
In the proof of \cite[Theorem 3.3]{CartisYan2014}, the construction of an optimal solution of the perturbed {\lp} problems relies on the existence of a strictly complementary solution, more exactly the strictly complementary partition{\footnoteref{fnt:strictlyCs} for the solution of the {\lp} problems; without this, \cite[Lemma 4.2]{CartisYan2014} will not hold and therefore the consequent Lemma 4.3 and the main prediction results, Theorems 4.4 -- 4.6, will not hold.

The main contributions in this paper lie on two directions.
\begin{itemize}
\item We extend the results to {\qp} without strictly complementary assumption, with all major prediction results having been reproduced for {\qp}. In particular, we present the result of preserving the active set from the aspect of a least-squares solution, which yields more general result. 
\item  The lack of strictly complementary solution leads to the analysis of the so-called `tripartition' (Section~\ref{sec:preservingOptActvQP}) instead of the optimal active and inactive partition~\cite{CartisYan2014}. We have proved that we can also predict the optimal tripartition of the original {\qp} problems by solving the perturbed ones.
\end{itemize}

\paragraph*{Structure of this paper}
In the following sections, we present the formulations of the perturbed {\qp} problems (Section~\ref{sec:cp4qp}) and their properties (Section~\ref{sec:properties_per_qp}).
We then derive theorems on predicting the optimal active set of a {\qp} problem without the strictly complementary assumption\footnoteref{fnt:strictlyCs} (Section~\ref{subsec-predict_opt_actv}); we also present results on predicting the optimal tripartition of a {\qp} problem (Section~\ref{subsec-predict_tripartition}). 
In Section~\ref{sec:numericsQP}, we first present the perturbed algorithm structure in Section~\ref{sec:perturbedAlg} and introduce the test problems in Section~\ref{sec:testproblems_qp}. 
In Section~\ref{sec:preidctionRatios_qp}, similarly to the linear case, we conduct numerical tests on the accuracy of the predicted optimal active set of the convex ({\qp}) prelims.
Then in Section~\ref{sec:solveSub_qp}, we predict the optimal active set, build a sub-problem by removing the active constraints and corresponding rows/columns in the problem data, $A$, $c$, and $H$,  solve the sub-problem using the active-set method and compare the number of active-set iterations. 
The feasibility error and relative difference between the optimal objective value of the sub-problem and that of the original problem are also measured; see~\eqref{eqv:feaErrorTest} and~\eqref{eqv:objErrorTest} for details.

\section{Controlled perturbations for quadratic programming problems}
\label{sec:cp4qp}
Consider the following pair of primal and dual convex {\qp} problems,
\begin{equation}
\label{eqv:orgQP}
\tag{QP}
\begin{array}{lll}
	\mbox{(Primal)} & \hspace{1cm} & \mbox{(Dual)}\\
	\begin{array}{cl}
		\displaystyle \min_{x} & \frac{1}{2}x^{T}Hx+c^{T}x\\
                 \mbox{s.t.}& Ax = b,\\
			& x \geq 0,
	\end{array} &
	\hspace{1em} &
	\begin{array}{cl}
		\displaystyle \max_{(x,y,s)} & b^{T}y - \frac{1}{2}x^{T}Hx\\
		\mbox{s.t.} & A^{T}y + s - Hx= c,\\
	               & y \text{ free},\,\, s \geq 0, 
	\end{array}
\end{array}
\end{equation}
where $H \in \Real^{n \times n}$ is symmetric positive semi-definite, $A \in \Real^{m \times n}$ with $m \leq n$, $y$, $b \in \Real^{m}$ and $x$, $s$, $c \in \Real^{n}$. When $H \equiv 0$, these problems reduce to the {\lp} problems.

We enlarge the feasible set of the~\eqref{eqv:orgQP}  problems by enlarging the nonnegativity constraints in~\eqref{eqv:orgQP} and consider the following perturbed problems,
\begin{equation}
\label{eqv:perQP}
\tag{QP$_{\lambda}$}
\resize{0.86}{
\begin{array}{lll}
	\mbox{(Primal)} & \hspace{1cm} & \mbox{(Dual)}\\
	\begin{array}{cl}
                 \min_{x} & \frac{1}{2}(x+\lambda)^{T}H(x+\lambda) \\
                               & +(c+(I-H)\lambda)^{T}(x+\lambda)\\
                 \mbox{s.t.}& Ax = b,\\
                            & x \geq -\lambda,
                \end{array} &
                \hspace{0em} &
                \begin{array}{cl}
                \max_{(x,y,s)} & (b+A\lambda)^{T}y \\ 
                			     & - \frac{1}{2}(x+\lambda)^{T}H(x+\lambda)\\
		\mbox{s.t.} & A^{T}y + s - Hx= c,\\
	               & y \text{ free},\,\, s \geq -\lambda,
	\end{array}
\end{array} }
\end{equation}
where $\lambda \in \Real^{n}$ and $\lambda \geq 0$. 
Note that if $\lambda \equiv 0$,~\eqref{eqv:perQP} is equivalent to~\eqref{eqv:orgQP}.
By formulating the Lagrangian dual~\cite{Bertsekas} of the primal (dual) problem in~\eqref{eqv:perQP}, it is straightforward to show the following result.
\begin{proposition} 
The two problems in~\eqref{eqv:perQP} are dual to each other.
\end{proposition}

We denote the set of {\em strictly feasible points} of~\eqref{eqv:perQP}
\begin{equation}
\label{eqv:perturbed_strcitly_fea_set}
	{\sfspqp} = \{(x,y,s) \,|\, Ax=b, \, A^{T}y+s-Hx = c, \, x+\lambda>0, \, s+\lambda>0\}.
\end{equation}
${\sfspqp}$ coincides with the strictly feasible set  of~\eqref{eqv:orgQP} if $\lambda \equiv 0$.

According to~\cite[Theorem 12.1]{Nocedal2006}, we derive the {\kkt} conditions for~\eqref{eqv:perQP},  
\begin{subequations}
\label{eqv:kkt_perturbedQP}
	\begin{align}
		Ax & = b,\label{eqv:fea-primal}\\
		A^{T}y+s-Hx & = c, \label{eqv:fea-dual}\\
		(X+\Lambda)(S+\Lambda)e & = 0, \label{eqv:kkt_perturbedQP_cpc}\\
		(x+\lambda,s+\lambda) & \geq  0 \label{eqv:kkt_perturbedQP-nonneg},
	\end{align}
\end{subequations}
where $\Lambda$ is a diagonal matrix with the entries of $\lambda$ on the diagonal and $e$ is a vector of ones. 
Any primal-dual pair $\pntp$ is an optimal solution of~\eqref{eqv:perQP} if and only if it satisfies~\eqref{eqv:kkt_perturbedQP}.  If $\lambda \equiv 0$, \eqref{eqv:kkt_perturbedQP} represents the {\kkt} conditions for~\eqref{eqv:orgQP}. 

\paragraph*{Equivalent formulation of~\eqref{eqv:perQP}}
Let $ p=x+\lambda$ and $q=s+\lambda$. Then we can rewrite~\eqref{eqv:perQP} as follows,
\begin{equation}
\label{eqv:perQP_pq}
\begin{array}{lll}
	\mbox{(Primal)} & \hspace{1cm} & \mbox{(Dual)}\\
	\begin{array}{cl}
                 \min_{p} & \frac{1}{2}p^{T} H p +\hat{c}_{\lambda}^{T}p\\
                 \mbox{s.t.}& Ap = \hat{b}_{\lambda},\\
                            & p \geq 0,
                \end{array} &
                \hspace{0em} &
                \begin{array}{cl}
                \max_{(p,y,q)} & \hat{b}_{\lambda}^{T}y - \frac{1}{2}p^{T} H p\\
		\mbox{s.t.} & A^{T}y + q - H p= \hat{c}_{\lambda},\\
	               & y \text{ free},\,\, q \geq 0,
	\end{array}
\end{array} 
\end{equation}
where 
\begin{equation}
\label{eqv:b_hat_c_hat}
	\hat{b}_{\lambda} = b+A\lambda \quad\text{and}\quad \hat{c}_{\lambda} = c + (I-H)\lambda.
\end{equation} 
Formulating the {\kkt} conditions of~\eqref{eqv:perQP_pq} and comparing them with~\eqref{eqv:kkt_perturbedQP}, we have the following result.
\begin{proposition}
$\optSollambda$ is an optimal solution of~\eqref{eqv:perQP} with some $\lambda \geq 0$ if and only if $(p^{*}_{\lambda},y^{*}_{\lambda},q^{*}_{\lambda})$, with $ p^{*}_{\lambda} = x^{*}_{\lambda} + \lambda$ and $ q^{*}_{\lambda} = s^{*}_{\lambda} + \lambda$, is a solution of~\eqref{eqv:perQP_pq}. 
\end{proposition}

\paragraph*{The central path of~\eqref{eqv:perQP}}
Following~\cite[Chapter 11]{Boyd}, we derive the central path equations for~\eqref{eqv:perQP}, namely
\begin{subequations}
\label{eqv:perturbedCentralPathQP}
\begin{align}
		Ax & =b,\\
		A^{T}y+s -Hx& =c, \\
		(X+\Lambda)(S+\Lambda)e & = \mu \, e, \label{eqv:perturbedCentralPathQP_CP}\\
		(x+\lambda,s+\lambda) & > 0,
\end{align}
\end{subequations}
where $\mu > 0$ is the barrier parameter for~\eqref{eqv:perQP}.  Note that \eqref{eqv:perturbedCentralPathQP} represents the central path equations for~\eqref{eqv:orgQP} when $\lambda \equiv 0$.
The central path of~\eqref{eqv:perQP} is well defined under mild assumptions, including
\begin{equation}
\label{asm:full_row_rank}
\text{\bf Assumption: } \hspace*{20ex}	A  \text{ has full row rank } m. \hspace*{20ex} 
\end{equation}
Under this assumption, Monteiro and Adler~\cite{Monteiro1989QP} show that the central path of a {\qp} problem exists if its strictly feasible set is nonempty. 
From this statement and considering the equivalent form~\eqref{eqv:perQP_pq} of~\eqref{eqv:perQP}, it follows that the central path of~\eqref{eqv:perQP} exists if its strictly feasible set $\sfspqp$ in~\eqref{eqv:perturbed_strcitly_fea_set} is nonempty. Thus we can draw the same conclusion as in the {\lp} case~\cite[Lemma 2.1]{CartisYan2014}, that given $\lambda > 0$, the existence of the perturbed central path requires weaker assumptions compared to those for the central path of~\eqref{eqv:orgQP}, because $\sfspqp$ is nonempty if~\eqref{eqv:orgQP} has a nonempty primal-dual feasible set.

\section{Properties of the perturbed quadratic programming  problems}
\label{sec:properties_per_qp} 
\subsection{Perfect and relaxed perturbations}
For the {\lp} case, we know that the optimal solution of the original problems can lie on or near the central path of the perturbed problems \cite[Section 3.1]{CartisYan2014}. Following exactly the same approach, we can verify that these results also hold for {\qp}.
\begin{theorem}[Existence of `perfect' perturbations for {\qp}]
\label{thm:ExistenceOfPerfectPerturbationsForQP}
\hfill\\
Assume~\eqref{asm:full_row_rank} holds and $\optSol$ is a solution of \eqref{eqv:orgQP}. Let $\hat{\mu} > 0$. Then there exist perturbations 
\[
	\hat{\lambda} = \hat{\lambda}(x^{*},s^{*},\hat{\mu}) > 0,
\]
such that the perturbed central path~\eqref{eqv:perturbedCentralPathQP} with $\lambda = \hat{\lambda}$~passes through  ${\optSol}$ exactly when $\mu = \hat{\mu}$.
\end{theorem}

\begin{theorem}[Existence of relaxed perturbations for {\qp}]
\label{thm:ExistenceOfRelaxedPerturbationsForQP}
\hfill \\
Assume $\optSol$ is a solution of \eqref{eqv:orgQP}. Let $\hat{\mu} > 0$ and $\xi \in (0,1)$. 
Then there exist constants $\hat{\lambda}_{L} = \hat{\lambda}_{L}(x^{*},s^{*},\hat{\mu},\xi)>0$, and $\hat{\lambda}_{U} = \hat{\lambda}_{U}(x^{*},s^{*},\hat{\mu},\xi)>0,$
such that for $\hat{\lambda}_{L} \leq \lambda \leq \hat{\lambda}_{U}$, $\optSol$ is strictly feasible for~\eqref{eqv:orgQP} and satisfies
\[
\xi \hat{\mu} e \leq (X^{*}+\Lambda)(S^{*}+\Lambda)e\leq \frac{1}{\xi} \hat{\mu} e. 
\]
\end{theorem}

Intuitively, these existence theorems imply that when the perturbations are chosen properly, the perturbed central path may pass or get very close to the original optimal solution. Thus 
we have the hope that from the iterates which follow the perturbed central path, we may be able to get enough information about the original optimal solution, so as to predict the optimal active set of the original problem.

\subsection{Preserving the  optimal active sets and tripartitions}
\label{sec:preservingOptActvQP}
Let $\optSol$ be a solution of~\eqref{eqv:orgQP} and define
\begin{equation}
\label{eqv:optActv_qp_x_s}
\begin{array}{rcl}
	\Actvqp(x^{*}) 	& = & \left\{ i \in \{1,\ldots,n\} \,|\, x^{*}_{i} = 0\right\},  \\
	\Iactvqp(x^{*})   & =  & \{ i \in \{1,\ldots,n\} \,|\, x^{*}_{i} > 0\}, \\
	\Actvqps(s^{*}) 	& = & \left\{ i \in \{1,\ldots,n\} \,|\, s^{*}_{i} = 0\right\}, \\
	\Sactvqp(s^{*})  & = & \{ i \in \{1,\ldots,n\} \,|\,s^{*}_{i} > 0\},
\end{array}
\end{equation}
where $\Actvqp(x^{*})$ is the \emph{primal active set} of~\eqref{eqv:orgQP}, $\Iactvqp(x^{*})$ the \emph{primal inactive set}, $\Actvqps(s^{*})$ the \emph{dual active set} and $\Sactvqp(s^{*})$ the \emph{dual inactive set}. From the complementary condition~\eqref{eqv:kkt_perturbedQP_cpc}} with $\lambda = 0$, it is easy to verify that 
\begin{equation}
\label{eqv:set_relations_qp}
	\Sactvqp(s^{*}) \subseteq \Actvqp(x^{*}), \quad \Iactvqp(x^{*}) \subseteq \Actvqps(s^{*}) \quad \text{and} \quad \Iactvqp(x^{*}) \cap \Sactvqp(s^{*}) = \emptyset.
\end{equation}
Note that $\Actvqp(x^{*}) \cap \Actvqps(s^{*})$ may not be empty.

We also denote
\begin{equation}
 \label{eqv:tripttn_qp_T}
 \T(x^{*},s^{*}) = \{1,\ldots,n \} \setminus \left(\Sactvqp(s^{*}) \cup \Iactvqp(x^{*})\right),
\end{equation}
which represents the complement of the optimal primal and dual inactive sets. 
This and~\eqref{eqv:set_relations_qp} give us that 
\[
\Sactvqp(s^{*}) \cap \Iactvqp(x^{*})  =  \Sactvqp(s^{*}) \cap \T(x^{*},s^{*}) = \Iactvqp(x^{*}) \cap \T(x^{*},s^{*}) = \emptyset,
\] 
and the union of them is the full index set, namely, $\Sactvqp(s^{*})$, $\Iactvqp(x^{*})$ and $\T(x^{*},s^{*})$ form an \emph{optimal tripartition} of $\{1,\ldots,n\}$ for~\eqref{eqv:orgQP}. 
From the definition of $\T(x^{*},s^{*})$, we have $x^*_{i} = s^*_{i} = 0$  for any $i \in  \T(x^{*},s^{*})$ and thus it is also straightforward to verify 
\begin{equation*}
\label{eqv:actv_trippt}
\Actvqp(x^{*}) = \Sactvqp(s^{*}) \cup  \T(x^{*},s^{*})
\quad\text{and}\quad 
\Actvqps(s^{*}) = \Iactvqp(x^{*}) \cup \T(x^{*},s^{*}).
\end{equation*}

The primal-dual pair in \eqref{eqv:orgQP} always has a \emph{maximal complementary solution}, at which the number of  positive components of $x^{*} + s^{*}$ is maximised~\cite{guler1993convergence}. 
Even at a maximal complementary solution,  $\T(x^{*},s^{*}) $ may not be empty because of the absence of the Goldman--Tucker Theorem for~\eqref{eqv:orgQP}.
Note that $\left( \Sactvqp(s^{*}),\Iactvqp(x^{*}), \T(x^{*},s^{*}) \right)$ forms a tripartition at any solution $\optSol$ of~\eqref{eqv:orgQP} but it may be different at different solutions; the tripartitions are only guaranteed to be invariant at maximal complementary solutions~\cite[Theorem 1.18]{ye2011interior}.

\paragraph*{Preserving the  optimal active sets}
Similarly, given a primal-dual pair $(x,y,s)$ for~\eqref{eqv:perQP}, we define the following sets
\begin{equation}
\label{eqv:perturbed_actv_qp_x_s}
\begin{array}{rl}
	\Actvqp_{\lambda}(x) = \left\{ i \in \{1,\ldots,n\} \,|\, x_{i} = -\lambda \right\},   & 	 \Iactvqp_{\lambda}(x)  =  \{ i \in \{1,\ldots,n\} \,|\,  x_{i} > -\lambda \}, \\
	\Actvqps_{\lambda}(s) = \left\{ i \in \{1,\ldots,n\} \,|\, s_{i} = -\lambda \right\}, &	 \Sactvqp_{\lambda}(s) =  \{ i \in \{1,\ldots,n\} \,|\,  s_{i} > -\lambda \}.
\end{array} 
\end{equation}

In the following theorem, we show that there exists  a primal-dual pair of points which is close to the optimal solution of~\eqref{eqv:perQP} and the corresponding active and inactive sets at this point are the same as the optimal active and inactive sets at an optimal solution of~\eqref{eqv:orgQP}.

\begin{theorem}
\label{thm:preserv_actv_qp}
Assume $\optSol$ is an optimal solution of~\eqref{eqv:orgQP}. Then there exist a positive constant $\hat{\lambda} = \hat{\lambda}(H, A,b,c,x^{*},s^{*})$,  a positive constant $C_{1} = C_{1} (H, A, x^{*},s^{*})$ and a primal-dual pair $(x,y,s)$ which satisfies~(\ref{eqv:kkt_perturbedQP_cpc},~\ref{eqv:kkt_perturbedQP-nonneg}) with $0 < \| \lambda \| < \hat{\lambda}$, such that 
\begin{equation}
\label{eqv:set_lambda_set_optsol}
	\Actvqp_{\lambda}(x) = \Actvqp(x^{*}), \quad  \Iactvqp_{\lambda}(x)  = \Iactvqp(x^{*}),  \quad \Actvqps_{\lambda}(s) = \Actvqps(s^{*}),   \quad \Sactvqp_{\lambda}(s) =  \Sactvqp(s^{*}), 
\end{equation}
and 
\begin{equation}
\label{eqv:fea_errer}
\max \left(\|Ax - b\|, \|A^{T}y + s - Hx - c\|\right) < C_{1}\| \lambda \|,
\end{equation}
where $\| \cdot \|$ is the Euclidean norm.
\end{theorem}

\begin{proof}
We work with the equivalent form~\eqref{eqv:perQP_pq} of the problems in~\eqref{eqv:perQP}. 
For convenience, for the rest of this proof, we neglect the dependency of the index sets on  $\optSol$ and use $\Actvqp$, $\Iactvqp$, $\Actvqps$ and $\Sactvqp$ to denote the partition of a matrix or a vector in accordance with the corresponding sets.
Since $\optSol$ is a solution of~\eqref{eqv:orgQP} and from~\eqref{eqv:optActv_qp_x_s}, we have
\begin{equation}
\label{eqv:origin_opt_AxbAtys_partiton_qp}
\begin{aligned}
    x^{*}_{\Actvqp} = 0, \quad x^{*}_{\Iactvqp} > 0 \quad \text{and} \quad s^{*}_{\Actvqps} = 0, \quad s^{*}_{\Sactvqp} > 0,  \\
    A_{\Iactvqp} x^{*}_{\Iactvqp} =  b, \quad  A^{T}_{\Actvqps}y^{*} - H_{\Actvqps\Iactvqp}x^{*}_{\Iactvqp}   =  c_{\Actvqps}, \quad A^{T}_{\Sactvqp}y^{*} + s_{\Sactvqp}^{*} - H_{\Sactvqp\Iactvqp}x^{*}_{\Iactvqp}   = c_{\Sactvqp},
\end{aligned}
\end{equation}
where $H_{XY}$ denotes $(H_{ij})_{i \in X, j \in Y}$.
We define a point $(\hat{p},\hat{y},\hat{q})$ to be 
\begin{equation}
\label{eqv:eqv:preserving_qp_primaldual}
\begin{aligned}
  \hat{p}_{\Actvqp}  =  0,  \quad \hat{p}_{\Iactvqp}  = x^{*}_{\Iactvqp} + \lambda_{\Iactvqp} + \hat{u}, \\
  \hat{y}  = y^{*}+ \hat{v}, \quad \hat{q}_{\Actvqps} = 0,  \quad \hat{q}_{\Sactvqp}  =  s^{*}_{\Sactvqp} + \lambda _{\Sactvqp} - H_{\Sactvqp\Actvqp}\lambda_{\Actvqp} - A_{\Sactvqp}^{T}\hat{v} + H_{\Sactvqp\Iactvqp} \hat{u}, 
\end{aligned}
\end{equation}
where $(\hat{u},\hat{v})$ is the minimal least-squares solution of 
\begin{equation}
\label{eqv:proof_ls}
	M
	\begin{bmatrix}
		u \\
		v
	\end{bmatrix}
	=
	W
	\begin{bmatrix}
		\lambda_{\Actvqp} \\
		\lambda_{\Actvqps}
	\end{bmatrix},
 \text{ with }
	M = 
	\begin{bmatrix}
		A_{\Iactvqp} & 0 \\
		-H_{\Actvqps\Iactvqp} & A^{T}_{\Actvqps}
	\end{bmatrix}
	\text{ and }
	W = \begin{bmatrix}
		A_{\Actvqp} & 0 \\
		-H_{\Actvqps\Actvqp} &  I_{\Actvqps}
	\end{bmatrix}.
\end{equation}
We are about to find conditions on $\lambda$ under which $\hat{p}_{\Iactvqp} > 0$ and $\hat{q}_{\Sactvqp} > 0$, and thus we can have~\eqref{eqv:kkt_perturbedQP_cpc},~\eqref{eqv:kkt_perturbedQP-nonneg} and~\eqref{eqv:set_lambda_set_optsol} hold. From~\cite[Theorem 2.2.1]{Stephen2009},
we have
\[
\begin{bmatrix}
\hat{u} \\
\hat{v}
\end{bmatrix}
=
M^{+}W
\begin{bmatrix}
	\lambda_{\Actvqp} \\
	\lambda_{\Actvqps}
\end{bmatrix},
\] 
where $M^{+}$ is the pseudo-inverse of $M$. 
This and norm properties give us
\begin{equation}
\label{eqv:preserving_qp_uvBounds}
\| (\hat{u},\hat{v}) \| \leq \| M^{+}W \| \cdot ( \|\lambda_{\Actvqp}\| + \|\lambda_{\Actvqps}\| ) \leq 2\|M^{+}W\|\cdot\|\lambda\|.
\end{equation}
Let 
\[
\hat{\lambda}  =  \min \left(\, 
\frac{\vmin{x^{*}_{\Iactvqp}}{s^{*}_{\Sactvqp}}}{2\|M^{+} W\|}   
\,,\,
\frac{\vmin{x^{*}_{\Iactvqp}}{s^{*}_{\Sactvqp}}}{ \|H_{\Sactvqp\Actvqp}\| + 2\left( \|A^{T}_{\Sactvqp}\| + \| H_{\Sactvqp\Iactvqp} \| \right) \|M^{+} W\| }   
\,\right),
\]
where $\vmin{x^{*}_{\Iactvqp}}{s^{*}_{\Sactvqp}}$ denotes the smallest elements of the vectors $x^{*}_{\Iactvqp}$ and $s^{*}_{\Sactvqp}$.
This, \eqref{eqv:eqv:preserving_qp_primaldual}, \eqref{eqv:preserving_qp_uvBounds}, $0 < \|\lambda\| < \hat{\lambda}$ and norm properties give us that
\[
\hat{p}_{\Iactvqp} \geq  x^{*}_{\Iactvqp} + \hat{u} \geq x^{*}_{\Iactvqp} - \|\hat{u}\|e_{\Iactvqp} > x^{*}_{\Iactvqp} - 2\hat{\lambda}\|M^{+}W\|e_{\Iactvqp} \geq 0
\] 
and
\begin{eqnarray*}
\hat{q}_{\Sactvqp} &\geq & s^{*}_{\Sactvqp} - H_{\Sactvqp\Actvqp}\lambda_{\Actvqp} - A_{\Sactvqp}^{T}\hat{v} + H_{\Sactvqp\Iactvqp} \hat{u} \\
	& \geq & s^{*}_{\Sactvqp} - ( \|H_{\Sactvqp\Actvqp}\| \cdot \|\lambda\| + \|A_{\Sactvqp}^{T}\| \cdot \|\hat{v}\| + \|H_{\Sactvqp\Iactvqp}\| \cdot \|\hat{u}\| ) \\
	& > & s^{*}_{\Sactvqp} - \left( \|H_{\Sactvqp\Actvqp}\| + 2\|M^{+}W\|  ( \|A_{\Sactvqp}^{T}\|+ \|H_{\Sactvqp\Iactvqp}\|) \right) \hat{\lambda}  \geq 0.
\end{eqnarray*}
It remains to prove~\eqref{eqv:fea_errer}.
From~\eqref{eqv:b_hat_c_hat},~\eqref{eqv:origin_opt_AxbAtys_partiton_qp}, and~\eqref{eqv:eqv:preserving_qp_primaldual}, we can verify
\begin{equation}
\label{eqv:feasibleConsQP_2}
\begin{array}{rll}
A\hat{p} - \hat{b}_{\lambda} & = & A_{\Iactvqp} \hat{u} - A_{\Actvqp}\lambda_{\Actvqp} \\[1ex]
			  & = &  \left( M_{1}
	\begin{bmatrix}
		\hat{u} \\
		\hat{v}
	\end{bmatrix} - 
	W_{1}
	\begin{bmatrix}
		\lambda_{\Actvqp} \\
		\lambda_{\Actvqps}
	\end{bmatrix} \right), \\[3ex]
A^{T}_{\Actvqps}\hat{y} + \hat{q}_{\Actvqps} - H_{\Actvqps\Iactvqp}\hat{p}_{\Iactvqp} - (\hat{c}_{\lambda})_{\Actvqps} & = &  -H_{\Actvqps\Iactvqp}\hat{u} + A^{T}_{\Actvqps}\hat{v} + H_{\Actvqps\Actvqp}\lambda_{\Actvqp} -\lambda_{\Actvqps}  \\[3ex]
& = &
\left( 
	M_{2}
	\begin{bmatrix}
		\hat{u} \\
		\hat{v}
	\end{bmatrix} - 
	W_{2}
	\begin{bmatrix}
		\lambda_{\Actvqp} \\
		\lambda_{\Actvqps}
	\end{bmatrix} \right), \\[3ex]
A^{T}_{\Sactvqp} \hat{y} + \hat{q}_{\Sactvqp} - H_{\Sactvqp\Iactvqp}\hat{p}_{\Iactvqp} - (\hat{c}_{\lambda})_{\Sactvqp} & =  & 0,
\end{array}
\end{equation}
where 
\[
M_{1} = 
\begin{bmatrix}
A_{\Iactvqp} & 0 
\end{bmatrix}, 
M_{2} = \begin{bmatrix}
-H_{\Actvqps\Iactvqp} & A^{T}_{\Actvqps}
	\end{bmatrix}, 
W_{1} = \begin{bmatrix}
	A_{\Actvqp} & 0
\end{bmatrix}
\quad\text{and}\quad
W_{2} =\begin{bmatrix}
		-H_{\Actvqps\Actvqp} & I_{\Actvqps}
	\end{bmatrix}.
\]
Since $(\hat{u},\hat{v})$ is the least-squares solution of~\eqref{eqv:proof_ls}, 
\[
M = \begin{bmatrix}
M_{1} \\
M_{2} 
\end{bmatrix}
 \quad\text{and}\quad 
W = \begin{bmatrix}
W_{1} \\
W_{2}
\end{bmatrix},
\]
we have
\begin{equation*}
\label{eqv:preserving_qp_FeasBoudns}
\begin{aligned}
\|A\hat{p} - \hat{b}_{\lambda}\| 
\leq \left\| M
	\begin{bmatrix}
		\hat{u} \\
		\hat{v}
	\end{bmatrix} - 
	W
	\begin{bmatrix}
		\lambda_{\Actvqp} \\
		\lambda_{\Actvqps}
	\end{bmatrix} \right\|
\leq \left\|
	W
	\begin{bmatrix}
		\lambda_{\Actvqp} \\
		\lambda_{\Actvqps}
	\end{bmatrix} \right\|
\leq 2\|W\| \|\lambda\|, \\
\|A^{T}_{\Actvqps}\hat{y} + \hat{q}_{\Actvqps} - H_{\Actvqps\Iactvqp}\hat{p}_{\Iactvqp} - (\hat{c}_{\lambda})_{\Actvqps}\|
\leq  \left\| M
	\begin{bmatrix}
		\hat{u} \\
		\hat{v}
	\end{bmatrix} - 
	W
	\begin{bmatrix}
		\lambda_{\Actvqp} \\
		\lambda_{\Actvqps}
	\end{bmatrix} \right\|
\leq \left\|
	W
	\begin{bmatrix}
		\lambda_{\Actvqp} \\
		\lambda_{\Actvqps}
	\end{bmatrix} \right\|
\leq 2\|W\| \|\lambda\|.
\end{aligned}
\end{equation*}
This and~\eqref{eqv:feasibleConsQP_2} imply that 
\begin{equation}
\label{eqv:bound_on_feasiblity_qp}
	\max\left(\|A\hat{p} - \hat{b}_{\lambda}\|, \|A^{T}\hat{y} + \hat{q} - H\hat{p} - \hat{c}_{\lambda}\|\right) \leq 2 \|W\|\|\lambda\|.
\end{equation}
\end{proof}

\paragraph{Remarks on Theorem~\ref{thm:preserv_actv_qp}}
\label{rem-preser_actv_qp}
\begin{itemize}
\item The point $(x, y, s)$ satisfies the bound~\eqref{eqv:kkt_perturbedQP-nonneg} on $(x, s)$ and the complementary condition~\eqref{eqv:kkt_perturbedQP_cpc}. Thus
the error~\eqref{eqv:fea_errer} in the equality constraints~(\ref{eqv:fea-primal},~\ref{eqv:fea-dual}) also bounds the `distance' between $(x, y, s)$ and the optimal solution set of~\eqref{eqv:perQP}. This feasibility error~\eqref{eqv:fea_errer} goes to 0 as $\lambda \to 0$, and so primal and dual feasibility can be approximately achieved.
Note that, the feasibility error comes from the residual of the least problem~\eqref{eqv:proof_ls}, in other words, if~\eqref{eqv:proof_ls} has a solution, $(x, y, s)$ will be an optimal solution of~\eqref{eqv:perQP} with $\lambda > 0$, at which the primal-dual active sets of~\eqref{eqv:perQP} are the same as the original~\eqref{eqv:orgQP}.
\item Relation \eqref{eqv:bound_on_feasiblity_qp} gives an upper bound on the feasibility constraints of the equivalent form~\eqref{eqv:perQP_pq} of~\eqref{eqv:perQP}. Setting $\hat{x} = \hat{p} - \lambda$ and $\hat{s} = \hat{q} - \lambda$, we can see this bound is also an upper bound for the feasibility constraints of~\eqref{eqv:orgQP}.
\end{itemize}

\paragraph*{Preserving the optimal tripartition}
In~\eqref{eqv:tripttn_qp_T}, we have defined the complement of the optimal primal and dual inactive sets. Similarly, we denote
\begin{equation}
 \label{eqv:perturbed_tripttn_qp_T}
 \T_{\lambda}(x,s) = \{1,\ldots,n \} \setminus \left(\Sactvqp_{\lambda}(s) \cup \Iactvqp_{\lambda}(x)\right),
\end{equation}
where $\Sactvqp_{\lambda}(s)$ and $\Iactvqp_{\lambda}(x)$ are defined in~\eqref{eqv:perturbed_actv_qp_x_s}.
Note that without the complementary condition~\eqref{eqv:kkt_perturbedQP_cpc}, $(\Sactvqp_{\lambda}(s),\Iactvqp_{\lambda}(x),  \T_{\lambda}(x,s))$ may not form a tripartition of the full index set.
In the following corollary,  we show that under certain conditions on the perturbations, there exists a primal-dual pair which is close to (ultimately in) the solution set of~\eqref{eqv:perQP}, such that $(\Sactvqp_{\lambda}(s),\Iactvqp_{\lambda}(x),  \T_{\lambda}(x,s))$ forms a tripartition and it is the same as the tripartition $(\Sactvqp(s^{*}),\Iactvqp(x^{*}),\T(x^{*},s^{*}))$ at an optimal solution $\optSol$ of~\eqref{eqv:orgQP}.

\begin{corollary}
\label{col-preserv_tripttn_qp}
Assume $\optSol$ is an optimal solution of~\eqref{eqv:orgQP}. Then there exist a positive constant $\hat{\lambda} = \hat{\lambda}(H,A,b,c,x^{*},s^{*})$,  a positive constant $C_{1} = C_{1} (H,A, x^{*},s^{*})$ and a primal-dual pair $(x,y,s)$ which satisfies~(\ref{eqv:kkt_perturbedQP_cpc},~\ref{eqv:kkt_perturbedQP-nonneg}) with $0 < \| \lambda \| < \hat{\lambda}$, such that $(\Sactvqp_{\lambda}(s),\Iactvqp_{\lambda}(x),  \T_{\lambda}(x,s))$  forms a tripartition of $\{1,\ldots,n\}$ and is the same as the partition $(\Sactvqp(s^{*}),\Iactvqp(x^{*}),\T(x^{*},s^{*}))$ for the original~\eqref{eqv:orgQP} with~\eqref{eqv:fea_errer} satisfied,  
where $\Sactvqp(s^{*})$ and $\Iactvqp(x^{*})$ are defined in~\eqref{eqv:optActv_qp_x_s}, $\T(x^{*},s^{*})$ in~\eqref{eqv:tripttn_qp_T}, $\Sactvqp_{\lambda}(s)$ and $\Iactvqp_{\lambda}(x)$ in~\eqref{eqv:perturbed_actv_qp_x_s} and $\T_{\lambda}(x,s)$ in~\eqref{eqv:perturbed_tripttn_qp_T}.
\end{corollary}
\begin{proof}
Recalling the definitions of  $\T(x^{*},s^{*})$ and $\T_{\lambda}(x,s)$, the results follow from Theorem~\ref{thm:preserv_actv_qp}.
\end{proof}

Corollary~\ref{col-preserv_tripttn_qp} shows that under the same conditions for Theorem~\ref{thm:preserv_actv_qp}, there exists a point that is close to the solution set of the perturbed problems and preserves the optimal tripartition of the original {\qp}.  This point can be an optimal solution of~\eqref{eqv:perQP} as well.

\section{Active-set prediction for~\eqref{eqv:orgQP} using perturbations}
\label{sec:predict_actv_qp}
We first introduce an error bound for~\eqref{eqv:orgQP} to measure the distance of a point to the solution set of~\eqref{eqv:orgQP}. We have derived an error bound for {\lp} in~\cite[Lemma 4.1]{CartisYan2014} and the following lemma is its extension to {\qp}.

\begin{lemma}[Error bound for~\eqref{eqv:orgQP}]
\label{lem-eb_qp}
Let $\pntp \in \sfspqp$, where $\sfspqp$ is defined in~\eqref{eqv:perturbed_strcitly_fea_set}, and $\lambda \geq 0$. Then there exists an optimal solution $\optSol$ of~\eqref{eqv:orgQP} such that  
\begin{equation}
\label{eqv:eb-qp}
	\| x - x^{*}\| \leq \tau_{p} (r(x,s) + w(x,s)) \quad \text{and} \quad \|s-s^{*}\| \leq \tau_{d} (r(x,s) + w(x,s)),
\end{equation}
where $\tau_{p}$ and $\tau_{q}$ are problem-dependent constants independent of $\pntp$ and $\optSol$, and
\begin{equation}
\label{eqv:error_bounds_feasible_x_s_qp}
	r( x, s )= \|\cmin{ x }{ s }\|
\quad \text{and} \quad
w( x, s ) = \|(-x, -s, x^{T} s )_{+}\|,
\end{equation}
and where $\cmin{x}{s} = \left(\, \min(x_{i},s_{i}) \,\right)_{ i = 1,\ldots, n }$ and $(x)_{+} = \left(\, \max(x_{i},0) \,\right)_{ i = 1,\ldots, n }$.
\end{lemma}

See~\ref{apd-proof_lemma_eb_qp} for the proof of this lemma.

We define a symmetric neighbourhood~\cite{Gondzio25year} of the perturbed central path~\eqref{eqv:perturbedCentralPathQP},
\begin{equation}
\label{eqv:symmetric_neighbourhood}
{\Nsymp} = \left\{\, (x,y,s) \in {\sfspqp}\,|\, \gamma\mup \leq  (x_{i}+\lambda_{i})(s_{i} + \lambda_{i}) \leq \frac{\mup}{\gamma},  \, i = 1,\ldots,n\,\,\right\}, 
\end{equation}
where $\gamma \in (0,1)$ and $\mup$ is defined as
\begin{equation}
\label{eqv:mup}
\mup = \frac{(x+\lambda)^{T}(s+\lambda)}{n}.
\end{equation}
In the following analysis of predicting  the optimal active set (Section~\ref{subsec-predict_opt_actv}) and tripartition (Section~\ref{subsec-predict_tripartition}), we always consider points in this neighbourhood. 

\begin{lemma}
\label{lem-upper_bound_norm_x_xstar_s_sstar}
Let $\pntp \in \Nsymp$~\eqref{eqv:symmetric_neighbourhood} for some $\lambda \geq 0$ and $\mup$ defined in~\eqref{eqv:mup}. Then there exists a solution $\optSol$ of~\eqref{eqv:orgQP} and problem-dependent constants $\tau_{p}$ and $\tau_{d}$ that are independent of $\pntp$ and $\optSol$, such that
\begin{equation}
\label{eqv:bounds_x_xstar_on_mu_lambda_qp}
\begin{aligned}
	\|x-x^{*}\| < \tau_{p} \left( C_{2} \sqrt{\mup} \max(\sqrt{\mup}, 1) + 4 \|\lambda\| \max\left(\|\lambda\|,1\right) \right), \\
	\|s-s^{*}\| < \tau_{d} \left( C_{2} \sqrt{\mup} \max(\sqrt{\mup}, 1) + 4 \|\lambda\| \max\left(\|\lambda\|,1\right) \right),
\end{aligned}
\end{equation}
where  
\begin{equation}
\label{eqv:C4_qp}
	C_{2}  = \sqrt{\frac{n}{\gamma}} + n .
\end{equation} 
\end{lemma}
\begin{proof}
Following the same proof of \cite[Lemma 4.3]{CartisYan2014}, we  have
\begin{equation}
\label{eqv:proof_w_qp}
		w(x,s) \leq n \mup + 2\|\lambda\| + \|\lambda\|^{2}.
\end{equation}
It remains to find an upper bound for $r(x,s)$ in~\eqref{eqv:error_bounds_feasible_x_s_qp}.
Since  $(x_{i} + \lambda_{i})(s_{i} + \lambda_{i}) \leq \frac{1}{\gamma} \mup$, 
if $x_{i} + \lambda_{i}  \leq s_{i} + \lambda_{i}  $, we have
\[
0<x_{i} + \lambda_{i} \leq \frac{\mup}{\gamma(s_{i} + \lambda_{i})} \leq \frac{\mup}{\gamma(x_{i} + \lambda_{i})}, 
\]
namely 
$
0 < x_{i} + \lambda_{i} \leq \sqrt{\frac{\mup}{\gamma}}.
$
Similarly if $x_{i} + \lambda_{i}  > s_{i} + \lambda_{i}  $, we also have 
$
0 < s_{i} + \lambda_{i} < \sqrt{\frac{\mup}{\gamma}}.
$
Thus
$
0 < \cmin{x+\lambda}{s+\lambda}\leq \sqrt{\frac{\mup}{\gamma}} e.  
$
So from~\eqref{eqv:error_bounds_feasible_x_s_qp} we have
\begin{equation}
\label{eqv:proof_r_qp}
\begin{array}{rcl}
r(x,s) & =  &\|\cmin{x+\lambda}{s+\lambda} - \lambda\|  \\
	& \leq & \|\cmin{x+\lambda}{s+\lambda}\| + \|\lambda\| \\ 
	&  \leq & \sqrt{\frac{n\mup}{\gamma}} + \|\lambda\|.
\end{array}
\end{equation}
The bounds in~\eqref{eqv:bounds_x_xstar_on_mu_lambda_qp} follow from~\eqref{eqv:eb-qp},~\eqref{eqv:proof_w_qp}, and~\eqref{eqv:proof_r_qp}.
\end{proof}

\subsection{Predicting the  original optimal active set}
\label{subsec-predict_opt_actv}
Let 
\begin{equation}
\label{eqv:predictedSets}
\begin{array}{rcl}
	\pacs (x)  & = & \left\{ i \in \{1,\ldots, n\} \, | \, x_{i}  <  C \right\}, \\
	\psas  (s) & = & \left\{ i \in \{1,\ldots, n\} \, | \, s_{i} \geq C \right\}, 
\end{array}
\end{equation}
where $C$ is some constant threshold. We consider $\pacs(x)$ as the predicted active set and $\psas(s)$ the  predicted strongly  active set of~\eqref{eqv:orgQP} at the primal-dual pair $\pntp$. 

We show that prediction results  for {\lp} (Theorems 4.4 -- 4.6 in~\cite{CartisYan2014}) can be  extended to  the {\qp} case, namely, under certain conditions, the active sets $\Actvqp(x^{*})$ and $\Sactvqp(s^{*})$ at some solution $\optSol$ of~\eqref{eqv:orgQP} are bounded  well by $\psasqp(s)$ and $\pacsqp(x)$ below and above (Theorem~\ref{thm:bound_active_set_qp}), and under stricter conditions, the predicted active set $\pacsqp(x)$ is equivalent to $\Actvqp(x^{*})$ (Theorem~\ref{thm:predict_active_set_qp}) and the predicted strongly active set  $\psasqp(s)$ equivalent to $\Sactvqp(s^{*})$ (Theorem~\ref{thm:predict_strictly_active_set_qp}).

\begin{theorem}
\label{thm:bound_active_set_qp}
Let $C > 0$ and fix the vector of perturbations $\lambda$ such that
\begin{equation}
\label{eqv:thm_bound_lambda_included_qp}
	 0 < \|\lambda\| < \min \left( 1, \frac{C}{8 \max(\tau_{p},\tau_{d}) }  \right),
\end{equation}
where $\tau_{p}$ and $\tau_{d}$ are problem-dependent constants in~\eqref{eqv:bounds_x_xstar_on_mu_lambda_qp}.
Let $\pntp \in \Nsymp$ with $\mup$ sufficiently small, namely,
\begin{equation}
\label{eqv:mu_bounds_included_qp}
	\mup < \min\left( 1,\,\left( \frac{C}{2 \max(\tau_{p},\tau_d) C_{2}} \right)^{2} \right),
\end{equation}
where $\Nsymp$ is defined in~\eqref{eqv:symmetric_neighbourhood}, $\mup$ in~\eqref{eqv:mup}in~\eqref{eqv:mup} and $C_{2} > 0$, defined in~\eqref{eqv:C4_qp}, is a problem-dependent constant. Then there exists a  solution $\optSol$ of~\eqref{eqv:orgQP} such that
\begin{equation}
\label{eqv:predictedActv_included_in_actuPartition}
	\psasqp(s) \subseteq \Sactvqp(s^{*}) \subseteq \Actvqp(x^{*}) \subseteq \pacsqp(x),
\end{equation}
where $\psasqp(s)$ and $\pacsqp(x)$  are defined in~\eqref{eqv:predictedSets}, $\Sactvqp(s^{*})$ and $\Actvqp(x^{*})$ in~\eqref{eqv:optActv_qp_x_s}.
\end{theorem}

\begin{proof}
We mimic the proof of \cite[Theorem 4.4]{CartisYan2014}. From the complementary condition in~\eqref{eqv:kkt_perturbedQP_cpc} with $\lambda = 0$, it is straightforward to derive $\Sactvqp(s^{*}) \subseteq \Actvqp(x^{*})$.
From $\|\lambda\| < 1$, $\mup < 1$ and~\eqref{eqv:bounds_x_xstar_on_mu_lambda_qp}, we have $\|x - x^{*}\| \leq \tau_p C_{2} \sqrt{\mup} + 4 \tau_p\|\lambda\|$.
This,~\eqref{eqv:thm_bound_lambda_included_qp},  and~\eqref{eqv:mu_bounds_included_qp} give us that when $i \in \Actvqp(x^{*})$, $x^{*}_{i} = 0$ and $ x_{i} \leq \tau_p C_{2} \sqrt{\mup} +4 \tau_p\|\lambda\| < C$. Thus $\Actvqp(x^{*}) \subseteq \pacsqp(x)$. Similarly, if $i \notin \Actvqp(x^{*})$, we have $s^{*}_{i} = 0$ and then $s_{i} \leq \tau_d C_{2} \sqrt{\mup} +4 \tau_d\|\lambda\| < C$, which implies $\psasqp(s) \subseteq \Actvqp(x^{*})$.
\end{proof}

\begin{theorem}
\label{thm:predict_active_set_qp}
Let
\begin{equation}
\label{eqv:psi_p_qp}
	\psi_{p} = \inf_{x^{*} \in \sspoqp}\min_{i \in \Iactvqp(x^{*})} x_{i}^{*},
\end{equation}
where $\sspoqp$ is the solution set of the primal problem in~\eqref{eqv:orgQP}, and $\Iactvqp(s^{*})$ is defined in~\eqref{eqv:optActv_qp_x_s}.
Assume $\psi_{p} > 0$. Fix $C$ and $\lambda$ such that 
\begin{equation}
\label{eqv:C_lambda_Predict_AIT_QP}
C = \frac{\psi_p}{2} \quad 
\text{and} \quad
0 < \|\lambda\| < \min\left( 1,  \frac{\psi_p}{16\max(\tau_{p},\tau_d)} \right).
\end{equation}
Let $ (x,y,s) \in  \Nsymp$ with $\mup$ sufficiently small, namely 
\begin{equation}
\label{eqv:mu_mumax_X_qp}
	\mup < \min\left( 1,\,\left( \frac{\psi_{p}}{4 \max(\tau_{p},\tau_d) C_{2}} \right)^{2} \right),
\end{equation}
where $\tau_{p}$ and $\tau_{d}$ are problem-dependent constants in~\eqref{eqv:bounds_x_xstar_on_mu_lambda_qp}, $\Nsymp$ is defined in~\eqref{eqv:symmetric_neighbourhood}, $\mup$ in~\eqref{eqv:mup} and $C_{2}$ in~\eqref{eqv:C4_qp}.
Then there exists an optimal solution $\optSol$ of~\eqref{eqv:orgQP}, such that 
\[
	\pacsqp(x) = \Actvqp(x^{*}),
\]
where $\pacsqp(x)$ is defined in~\eqref{eqv:predictedSets} and $\Actvqp(x^{*})$ in~\eqref{eqv:optActv_qp_x_s}.
\end{theorem}

\begin{proof}
Setting $C = \frac{\psi_{p}}{2}$ in Theorem~\ref{thm:bound_active_set_qp}, we have~\eqref{eqv:predictedActv_included_in_actuPartition}. It remains to prove $\pacsqp(x) \subseteq \Actvqp(x^{*})$. If $i \notin \Actvqp(x^{*})$, $i \in \Iactvqp(x^{*})$ and we have $x^{*}_{i} > 0$. Then from~\eqref{eqv:psi_p_qp},~\eqref{eqv:C_lambda_Predict_AIT_QP} and~\eqref{eqv:mu_mumax_X_qp}, 
$
x_{i} \geq x^{*}_{i} -  \tau_p C_{2} \sqrt{\mup} - 4\tau_p\|\lambda\| > \psi_{p} - \frac{\psi_{p}}{2} =C,
$ 
namely $i \notin \pacsqp(x)$. Thus $\pacsqp(x) \subseteq \Actvqp(x^{*})$.
\end{proof}

\begin{theorem}
\label{thm:predict_strictly_active_set_qp}
Let
\begin{equation}
\label{eqv:psi_d_qp}
	\psi_{d} = \inf_{(y^{*},s^{*}) \in \ssdoqp}\min_{i \in \Sactvqp(s^{*})} s_{i}^{*},
\end{equation}
where $\ssdoqp$ is the solution set of the primal problem in~\eqref{eqv:orgQP}, and $\Sactvqp(s^{*})$ is defined in~\eqref{eqv:optActv_qp_x_s}.
Assume $\psi_{d} > 0$. Fix $C$ and $\lambda$ such that 
\begin{equation}
\label{eqv:C_lambda_Predict_SACT_QP}
C = \frac{\psi_{d}}{2} \quad 
\text{and} \quad
0 < \|\lambda\| < \min\left( 1,  \frac{\psi_{d}}{16\max(\tau_{p},\tau_d)} \right).
\end{equation}
Let $ (x,y,s) \in  \Nsymp$ with $\mup$ sufficiently small, namely 
\begin{equation}
\label{eqv:mu_mumax_S_qp}
\mup < \min\left( 1,\,\left( \frac{\psi_{d}}{4 \max(\tau_{p},\tau_d) C_{2}} \right)^{2} \right),
\end{equation}
where $\tau_{p}$ and $\tau_{d}$ are problem-dependent constants in~\eqref{eqv:bounds_x_xstar_on_mu_lambda_qp}, $\Nsymp$ is defined in~\eqref{eqv:symmetric_neighbourhood}, $\mup$ in~\eqref{eqv:mup} and $C_{2}$ in~\eqref{eqv:C4_qp}.
Then there exists an optimal solution $\optSol$ of~\eqref{eqv:orgQP}, such that 
\[
	\psasqp(s) = \Sactvqp(s^{*})
\]
where $\psasqp(s)$ is defined in~\eqref{eqv:predictedSets} and $\Sactvqp(s^{*})$ in~\eqref{eqv:optActv_qp_x_s}.
\end{theorem}

\begin{proof}
Setting $C = \frac{\psi_{d}}{2}$ in Theorem~\ref{thm:bound_active_set_qp}, we have~\eqref{eqv:predictedActv_included_in_actuPartition}. It remains to prove that $\Sactvqp(s^{*}) \subseteq \psasqp(s)$. If $i \in \Sactvqp(s^{*})$, we have $s^{*}_{i} > 0$. Then from~\eqref{eqv:psi_d_qp},~\eqref{eqv:C_lambda_Predict_SACT_QP} and~\eqref{eqv:mu_mumax_S_qp}, 
$
s _{i} \geq s^{*}_{i} -  \tau_d C_{2} \sqrt{\mu} - 4\tau_d\|\lambda\| > \psi_{d} - \frac{\psi_{d}}{2} =C,
$ 
namely $i \in \psasqp(s)$. Thus $\Sactvqp(s^{*}) \subseteq \psasqp(s)$.
\end{proof}

\paragraph{Remarks on Theorems~\ref{thm:bound_active_set_qp}--\ref{thm:predict_strictly_active_set_qp}}
\begin{itemize}
\item The results for {\lp} (\cite[Theorems 4.4 -- 4.6]{CartisYan2014}) only require the primal-dual pair $\pntp$ to be in the strictly feasible set of the perturbed problem, but we need to restrict $\pntp$ to the symmetric neighbourhood defined in~\eqref{eqv:symmetric_neighbourhood} for the {\qp} case. This is a more restrictive condition  but essential to the proof of Lemma~\ref{lem-upper_bound_norm_x_xstar_s_sstar}.  The presence of $\sqrt{\mup}$ in~\eqref{eqv:bounds_x_xstar_on_mu_lambda_qp} leads to a squared term in the thresholds~\eqref{eqv:mu_bounds_included_qp},~\eqref{eqv:mu_mumax_X_qp} and~\eqref{eqv:mu_mumax_S_qp} for $\mup$, which implies that, comparing with the results for {\lp}, we may need to decrease $\mup$ further before we can predict the  optimal active set of a {\qp} problem.
\item Theorems~\ref{thm:bound_active_set_qp} shows that the predicted strongly active set is included in the active set and the active set is a subset of the predicted active set. The intersection of these two predictions can serve as an approximation of the  optimal active set, which is what we do in the implementation. Theorems~\ref{thm:predict_active_set_qp} and~\ref{thm:predict_strictly_active_set_qp} show that under certain conditions on the perturbations and duality gap, we could predict exactly the optimal active and strongly active sets at some optimal solution $\optSol$ of~\eqref{eqv:orgQP}. 
Similarly to the {\lp} case, the same quantities $\psi_p$ and $\psi_d$ are present in the theorems. When~\eqref{eqv:orgQP} has a unique primal (dual) solution,  $\psi_p > 0$ ($\psi_d > 0$).
But $\psi_p$ and $\psi_d$ are only  theoretical  constants and our implementation does not depend on their values.
\end{itemize}

\subsection{Predicting the original optimal tripartition}
\label{subsec-predict_tripartition}
Let 
\begin{equation}
\label{eqv:preidcted_iactv_zerosxs_qp}
\begin{array}{rcl}
	\piacsqp(x) & = & \{ i \in \{1,\ldots,n\} \,|\, x_{i} \geq C\}, \\
	\pzeroxs(x,s) & = & \{1,\ldots,n \} \setminus (\psasqp (s) \cup \piacsqp (x)),
\end{array}
\end{equation}
where $C$ is some constant threshold and $\psasqp(s)$ defined in~\eqref{eqv:predictedSets}. We consider $\left(\psasqp(s), \piacsqp(x), \pzeroxs(x,s)\right)$ as the prediction
of the optimal tripartition of~\eqref{eqv:orgQP} at the primal-dual pair $\pntp$. Note that $\left(\psasqp(s), \piacsqp(x), \pzeroxs(x,s)\right)$ may not be a tripartition for an arbitrary point as the complementary condition~\eqref{eqv:kkt_perturbedQP-nonneg} may not be satisfied and thus $\psasqp(s) \cap \piacsqp(x)$ could be nonempty. The following two theorems, Theorems~\ref{col-bound_tripartition_set_qp} and~\ref{col-predict_tripartition_qp}, show that, under certain conditions on $\mup$ and $\lambda$, we are able to predict part or the whole of the tripartition.
\begin{theorem}
\label{col-bound_tripartition_set_qp}
Let $C > 0$ and fix the perturbation $\lambda$ such that $\|\lambda\|$ satisfies~\eqref{eqv:thm_bound_lambda_included_qp}. 
Let $\pntp \in \Nsymp$ with $\mup$ sufficiently small, namely, $\mup$ satisfies~\eqref{eqv:mu_bounds_included_qp}.
Then there exists an optimal solution $\optSol$ of~\eqref{eqv:orgQP} such that
\begin{equation}
\label{eqv:predictedPartition_included_in_actuPartition}
	\piacsqp(x) \subseteq \Iactvqp(x^{*}),
	\quad
	\psasqp(s) \subseteq \Sactvqp(s^{*}),
	\quad\text{and}\quad
	\T(x^{*},s^{*}) \subseteq \pzeroxs(x,s),
\end{equation}
where $\Iactvqp(x^{*})$ and $\Sactvqp(s^{*})$ are defined in~\eqref{eqv:optActv_qp_x_s},  $\T(x^{*},s^{*})$ in~\eqref{eqv:tripttn_qp_T}, $\piacsqp(x)$ and $\pzeroxs(x,s)$ in~\eqref{eqv:preidcted_iactv_zerosxs_qp}, and  $\psasqp(s)$ in~\eqref{eqv:predictedSets}.
\end{theorem}
\begin{proof}
Theorem~\ref{thm:bound_active_set_qp} shows that $\psasqp(s) \subseteq \Sactvqp(s^{*})$.
From~\eqref{eqv:predictedActv_included_in_actuPartition}, we have $ \Actvqp(x^{*}) \subseteq \pacsqp(x)$.
This, $\piacsqp(x) = \{1,\ldots,n\} \setminus \pacsqp(x)$, and $\Iactvqp(x^{*}) = \{1,\ldots,n\} \setminus \Actvqp(x^{*})$, give us that $\piacsqp(x) \subseteq \Iactvqp(x^{*})$.
$\T(x^{*},s^{*}) \subseteq \pzeroxs(x,s)$ follows directly from~\eqref{eqv:tripttn_qp_T} and~\eqref{eqv:preidcted_iactv_zerosxs_qp}.
\end{proof}

\begin{theorem}
\label{col-predict_tripartition_qp}
Let
\begin{equation}
\label{eqv:psi_qp}
	\psi = \min\left( \inf_{x^{*} \in \sspoqp}\min_{i \in \Iactvqp(x^{*})} x_{i}^{*},\, \inf_{(y^{*},s^{*}) \in \ssdoqp}\min_{i \in \Sactvqp(s^{*})}  s_{i}^{*}       \right),
\end{equation}
where $\sspoqp$ is the solution set of the primal problem in~\eqref{eqv:orgQP},  $\ssdoqp$ is the solution set of the dual problem and $\Iactvqp(s^{*})$ and $\Sactvqp(s^{*})$ are defined in~\eqref{eqv:optActv_qp_x_s}.
Assume $\psi > 0$. Fix $C$ and $\lambda$ such that 
\begin{equation}
\label{eqv:C_lambda_Predict_Trip_QP}
C = \frac{\psi}{2} \quad 
\text{and} \quad
0 < \|\lambda\| < \min\left( 1,  \frac{\psi}{16\max(\tau_{p},\tau_d)} \right).
\end{equation}
Let $ (x,y,s) \in  \Nsymp$ with $\mup$ sufficiently small, namely 
\begin{equation}
\label{eqv:mu_maX_qp_trip}
0< \mup < \min\left( 1,\,\left( \frac{\psi}{4 \max(\tau_{p},\tau_d) C_{2}} \right)^{2} \right),
\end{equation}
where $\tau_{p}$ and $\tau_{d}$ are problem-dependent constants in~\eqref{eqv:bounds_x_xstar_on_mu_lambda_qp}, $\Nsymp$ is defined in~\eqref{eqv:symmetric_neighbourhood}, $\mup$ in~\eqref{eqv:mup} and $C_{2}$ in~\eqref{eqv:C4_qp}.
Then there exists an optimal solution $\optSol$ of~\eqref{eqv:orgQP}, such that 
\[
	\psasqp(s) = \Sactvqp(s^{*}), \quad \piacsqp(x) = \Iactvqp(x^{*})  \quad \text{and} \quad \pzeroxs(x,s) = \T(x^{*},s^{*}),
\]
where $\T(x^{*},s^{*})$ is defined in~\eqref{eqv:tripttn_qp_T}, $\psasqp(s)$ in~\eqref{eqv:predictedSets}, and $ \piacsqp(x)$ and $\pzeroxs(x,s) $ defined in~\eqref{eqv:preidcted_iactv_zerosxs_qp}.
\end{theorem}

\begin{proof}
Setting $C = \frac{\psi}{2}$ in Theorem~\ref{col-bound_tripartition_set_qp}, we have~\eqref{eqv:predictedPartition_included_in_actuPartition}. It remains to prove that
$\Iactvqp(x^{*}) \subseteq \piacsqp(x)$ and $ \Sactvqp(s^{*}) \subseteq \psasqp(s)$.
From~\eqref{eqv:psi_qp},~\eqref{eqv:C_lambda_Predict_Trip_QP} and~\eqref{eqv:mu_maX_qp_trip}, if $i \in \Iactvqp(x^{*})$, we have $x^{*}_{i} > 0$ and then $x_{i} \geq x^{*}_{i} -  \tau_p C_{2} \sqrt{\mup} - 4\tau_p\|\lambda\| > \psi - \frac{\psi}{2} =C$, namely $i \in \piacsqp(x)$. 
Thus $\Iactvqp(x^{*}) \subseteq \piacsqp(x)$.
Similarly, we can also have $ \Sactvqp(s^{*}) \subseteq \psasqp(s)$. 
Therefore $\pzeroxs(x,s)  = \T(x^{*},s^{*})$.
\end{proof}

\section{Numerical experiments for quadratic programming using perturbations}
\label{sec:numericsQP}
\subsection{The perturbed algorithm and its implementation}
\label{sec:perturbedAlg}
All numerical experiments in this section employ an infeasible primal-dual path-following {\ipm} applied to~\eqref{eqv:perQP} or~\eqref{eqv:orgQP}. 
The perturbed algorithm is summarised in \textbf{Algorithm~\ref{alg:perAlgQP}} and it is nothing but an infeasible {\ipm} applied to~\eqref{eqv:perQP} with possible shrinkage of the perturbations.

\begin{algorithm}[h]                     
\caption{The Perturbed Algorithm with Active-set Prediction for {\qp}}          
\label{alg:perAlgQP}    
{\normalsize                      
\begin{algorithmic}                    
	\STATE{%
	\textbf{Step 0:} choose perturbations $(\lambda^{0}, \phi^{0})>0$ and calculate  a Mehrotra starting point $(x^0,y^0,s^0)$;}%
    	\FOR{$k = 0,1,2,\ldots$} 
    		\STATE{%
		\textbf{Step 1:} solve the perturbed Newton system~\eqref{eqv:perturbedCentralPathQP} using the augmented system approach, namely
		\begin{align*}
	\begin{bmatrix}
	-H-D_{\lambda}^{-2} & A^{T} \\
	A         & 0       
	\end{bmatrix}
	\begin{bmatrix}
	 \Delta x^{k} \\
	 \Delta y^{k} 
	\end{bmatrix}
	  & = -
	 \begin{bmatrix}
	 R^{k}_{d} - (X^{k} + \Lambda^{k} )^{-1} R^{k}_{\mup}\\
	 R^{k}_{p}
	 \end{bmatrix}, \\
	 \Delta s^{k} & = - \left( X^{k} + \Lambda^{k} \right)^{-1} \left(  R^{k}_{\mup}  + \left(  S^{k} + \Phi^{k} \right) \Delta x^{k} \right),
	\end{align*}
	where $D_{\lambda} = \left(  S^{k} + \Phi^{k} \right)^{-\frac{1}{2}} \left( X^{k} + \Lambda^{k}  \right)^\frac{1}{2}$, $R^{k}_{p} = Ax^{k} - b$, $R^{k}_{d} = A^{T}y^{k} +s^{k} - Hx^{k}- c$, $R^{k}_{\mup}  = \left( X^{k}+\Lambda^{k} \right) 	\left( S^{k} + \Phi^{k} \right)e -  \sigma^{k}\mup^{k}e $,
	and where $\sigma^{k}  = \min (0.1, 100\mup^{k} )\in [ 0 ,1 ]$
	and 
	\begin{equation}
	\label{eqv:mupk}
	\mup^{k} = \frac{(x^{k} + \lambda^{k})^{T}(s^{k} + \phi^{k})}{n};
	\end{equation}}%
	\STATE{%
	\textbf{Step 2:} choose a fixed, close to 1, fraction of the stepsize to the nearest constraint boundary in the primal and dual space, respectively.
	Namely, 
$
\alpha_{p}^{k}  =  \min \left(  \bar{\alpha}  \min_{ i \, : \Delta x^{k}_{i} < 0 } \left(  \frac{-x^{k}_{i} - \lambda^{k}_{i}}{  \Delta x^{k}_{i} }  \right),\,  1 \right),
$  
and 
$
\alpha_{d}^{k}  =  \min \left(  \bar{\alpha}  \min_{ i \, : \Delta s^{k}_{i} < 0 } \left(  \frac{-s^{k}_{i} - \phi^{k}_{i}}{  \Delta s^{k}_{i} }  \right),\, 1 \right),
$
	where $\bar{\alpha} = 0.9995$; }%
	\STATE{%
	\textbf{Step 3:} update $ x^{k+1} = x^{k} + \alpha^{k}_{p} \,\Delta x^{k}$ and $(y^{k+1},s^{k+1}) = (y^{k} ,s^{k}) + \alpha^{k}_{d} \,( \Delta y^{k}, \Delta s^{k})$;
	}%
	\STATE{%
	\textbf{Step 4:} predict the  optimal active set of~\eqref{eqv:orgQP} and denote as $\Actv^{k}$;
	}%
	\STATE{%
	\textbf{Step 5:} terminate if some termination criterion is satisfied;
	}%
	\STATE{%
	\textbf{Step 6:} obtain $(\lambda^{k+1}, \phi^{k+1})$ possibly by shrinking $(\lambda^{k}, \phi^{k})$ so that $(x^{k+1} +  \lambda^{k+1}, s^{k+1} + \phi^{k+1}) >0$.
	}%
	\ENDFOR
\end{algorithmic} }
\end{algorithm}

\paragraph{Algorithm without perturbations for {\qp}}
For comparison in the numerical tests, we refer to the algorithm with no perturbations (Algorithm~\ref{alg:perAlgQP} with $\lambda = \phi = 0$) as \textbf{{\refstepcounter{algorithm}\label{alg:unperAlgQP}}Algorithm~{\thealgorithm}}.
We use the notation $\mu^{k}$, which is equivalent to $\mup^{k}$~\eqref{eqv:mupk} with $\lambda^{k} = \phi^{k} = 0$ for the duality gap for Algorithm~\ref{alg:unperAlgQP}. 
  
Most of the implementation details follow similarly to the {\lp} case unless specified. 
We apply the Mehrotra \textbf{starting point}~\cite{mehrotra} for both perturbed (Algorithm~\ref{alg:perAlgQP}) and unperturbed (Algorithm~\ref{alg:unperAlgQP}) algorithms.\footnote{Note that we modify Mehrotra's procedure and calculate a min-norm primal-dual feasible point for~\eqref{eqv:orgQP}, namely we replace $\tilde{s} = c-A^{T}\tilde{y}$ in \cite[(7.1)]{mehrotra} with $\tilde{s} = c-A^{T}\tilde{y} +Q\tilde{x}$. } 
We \textbf{shrink perturbations} according to the value of the smallest elements of the current iterate, for instance, at iteration $k$, we choose a fixed fraction of $\lambda^{k}$ when $\min(x^{k}) > 0$, otherwise we find a point on the line segment connecting $\lambda^{k}$ and $ - \min(x^{k})e$; similarly for $\phi^{k}$.  
The \textbf{initial perturbations} are set to $\lambda^{0} = \phi^{0}  = 10^{-3}e$ for all numerical tests. 
We utilise the same \textbf{active-set prediction} procedure proposed in \cite[Section 6.1]{CartisYan2014}, namely, we move the indices between the predicted active, predicted inactive, and undetermined sets, depending on whether the criteria $x^{k}_{i} < C$ and $s^{k}_{i} > C$ are satisfied (see Procedure~\ref{alg:actvPrediction} in~\ref{appnd:actvprocedure} for details). 
\textbf{Termination criteria} will be defined for each set of tests. Relative residual is also employed in the following tests to measure the distance from the iterates to the optimal solution set of~\eqref{eqv:perQP}, namely 
\begin{equation}
\label{eqv:resResidualQP}
	\mbox{Res}^{k}_{\lambda} =  \frac{|| \left( Ax^{k} - b, A^{T}y^{k} + s^{k} - Hx^{k} - c, \left( X^{k}+\Lambda^{k} \right) \left( S^{k} + \Phi^{k} \right)e \right)  ||_{\infty}}{1+\max \left(  || b ||_{\infty}, || c ||_{\infty}  \right)}.
\end{equation}

\subsection{Test problems}
\label{sec:testproblems_qp}
\paragraph{Randomly generated problems (QTS1)}\label{test-random_set_qp} 
We first randomly generate the number of constraints $m \in (10,200) $, the number of variables $n \in (20,500) $ and the matrix $A$ following the same procedure described in \cite[Section 6.2]{CartisYan2014} for generating random {\lp} test problems.
Then randomly generate a full rank square matrix $B \in \Real^{n \times n}$ and set the quadratic term $H = B'B$. Next we generate a triple $(x,y,s) \in \Real^{n}\times\Real^{m}\times\Real^{n}$ with $(x,s) \geq 0$ and density about $0.5$. Finally we obtain $b = Ax$ and $c = A^{T}y+s-Hx$. Thus $(x,y,s)$ is used as a feasible point for this problem.
50 problems are generated for this test set.

\paragraph{Randomly generated degenerate problems (QTS2)}\label{test-random_pd_degen_set_qp}
First generate $m$, $n$, $A$ and $H$ as for {\tspndegqp}. Apart from generating a feasible point as we do for {\tspndegqp}, we generate a primal-dual degenerate optimal solution here. Namely we generate a triple $(x,y,s)$ with $(x,s) \geq 0$, $x_{i}s_{i} = 0$ for all $i \in \{1,\ldots,n\}$ and the number of positive components of $x$ strictly less than $m$ and that of $s$ strictly less than $n-m$. Then we get $b$ and $c$ as for {\tspndegqp}.
50 problems are also generated for this test set.

\paragraph{Convex {\qp} test problems from Netlib~\cite{netlib} and Maros and Meszaros' test sets~\cite{maros1999repository} (QTS3)}\label{test-netlib_mm_convex} We choose 7 small problems from the Netlib {\lp} test set
and add the identity matrix  as the quadratic term. We also choose 13 small problems from Maros and Meszaros' convex {\qp} collection\footnote{%
\url{www.doc.ic.ac.uk/~im/\#DATA}.}. All test problems have been transformed to the form with only equality constraints and nonnegative bounds on $x$ by adding slack variables. The dimensions of the problems are small, namely $m < 200$ and $n<250$ including slack variables. For the full list of the problems, see Table~\ref{tab:qpCuterProblems}. Note that the problems whose names  start with `QP\_'  are obtained from {\sc netlib}.

\begin{table}[H]
\centering
\caption{Convex {\qp} test problems from Netlib and Maros and Meszaros' test set}
\label{tab:qpCuterProblems}
\begin{small}
\begin{tabular}{| lll || lll |}
\hline
Name & m & n & Name & m & n \\
\hline\hline
QP\_ADLITTLE &   55 &  137 & QP\_AFIRO &   27 &   51 \\
  QP\_BLEND &   74 &  114 &     QP\_SC50A &   49 &   77     \\
  QP\_SC50B &   48 &   76 &      QP\_SCAGR7 &  129 &  185 \\ 
QP\_SHARE2B &   96 &  162 &   CVXQP1\_S &  150 &  200 \\  
  CVXQP2\_S &  125 &  200 &     CVXQP3\_S &  175 &  200 \\   
     DUAL1 &   86 &  170 &           DUAL2 &   97 &  192 \\
     DUAL3 &  112 &  222 &          DUAL4 &   76 &  150 \\      
     HS118 &   44 &   59 &            HS21 &    3 &    5 \\
      HS51 &    3 &   10 &          HS53 &    8 &   10 \\    
      HS76 &    3 &    7 &       ZECEVIC2 &    4 &    6 \\
      \hline\hline
\end{tabular}
\end{small}
\end{table}

\subsection{On the accuracy of optimal active-set predictions}
\label{sec:preidctionRatios_qp}
Assume $\Actvqp^{k}$ is the predicted active set  and $\Actvqp(x^{*})$ the actual active set at a primal optimal solution $x^{*}$ of~\eqref{eqv:orgQP}.
To measure the accuracy of our predictions, we also make use of the three prediction ratios defined in~\cite[Section 6.2.2]{CartisYan2014}. Namely, 
\begin{itemize}
\item False-prediction ratio    $\displaystyle = \frac{|\,\Actvqp^{k} \,\setminus\, (\Actvqp^{k} \,\cap\, \Actvqp(x^{*})) \,|}{|\, \Actvqp^{k} \,\cup\, \Actvqp(x^{*}) \,|}$,
\item Missed-prediction ratio $\displaystyle = \frac{|\,\Actvqp(x^{*}) \,\setminus\, (\Actvqp^{k} \,\cap\, \Actvqp(x^{*}))  \,|}{|\,  \Actvqp^{k} \,\cup\, \Actvqp(x^{*}))  \,|}$,
\item Correction ratio             $\displaystyle = \frac{|\, \Actvqp^{k} \,\cap\, \Actvqp(x^{*})   \,|}{|\, \Actvqp^{k} \,\cup\, \Actvqp(x^{*}) \,|}$.
\end{itemize}
False-prediction, missed-prediction  and correction ratios measure the degree of incorrectly identified active constraints, the degree of incorrectly rejected active constraints and the accuracy of the prediction, respectively. 
It is clear that all the three ratios are between 0 and 1 and the correction ratio is 1 if the predicted set is the same as the actual  optimal active set. The main task for this test is to compare the three measures for Algorithms~\ref{alg:perAlgQP} and~\ref{alg:unperAlgQP}.

To measure and compare the accuracy of the predicted active sets, we terminate Algorithms~\ref{alg:perAlgQP} and~\ref{alg:unperAlgQP} at the same iteration, and compare the predicted active sets with the original optimal active set at a solution obtained from the active-set method and that at a maximal complementary solution (the analytic center of the solution set) from an interior point method.\footnote{%
We solve the problem using Matlab's {\qp} solver \emph{quadprog} with the `Algorithm' option set to interior point or active set and consider all variables of the optimal solution $x^{*}$ less than $10^{-5}$ as active. %
} 
These two original optimal active sets can be different.\footnote{\label{foo-actv_difference}%
The difference is about $5\%$ on average for problems in {\tspndegqp} and $30\%$ for problems in {\tspddegqp}.%
}
Through this test, we also try to answer which active sets (at a solution from the active-set solver or a maximal complementary solution) Algorithm~\ref{alg:perAlgQP} predicts.  We test on two test cases, random problems (\tspndegqp) and random degenerate problems ({\tspddegqp}).
\begin{figure}[h]
\fboxsep= 3.5pt
\fboxrule=0.5pt
\begin{minipage}[t]{0.46\linewidth}
\centering\fbox{
\includegraphics[width=\textwidth]
{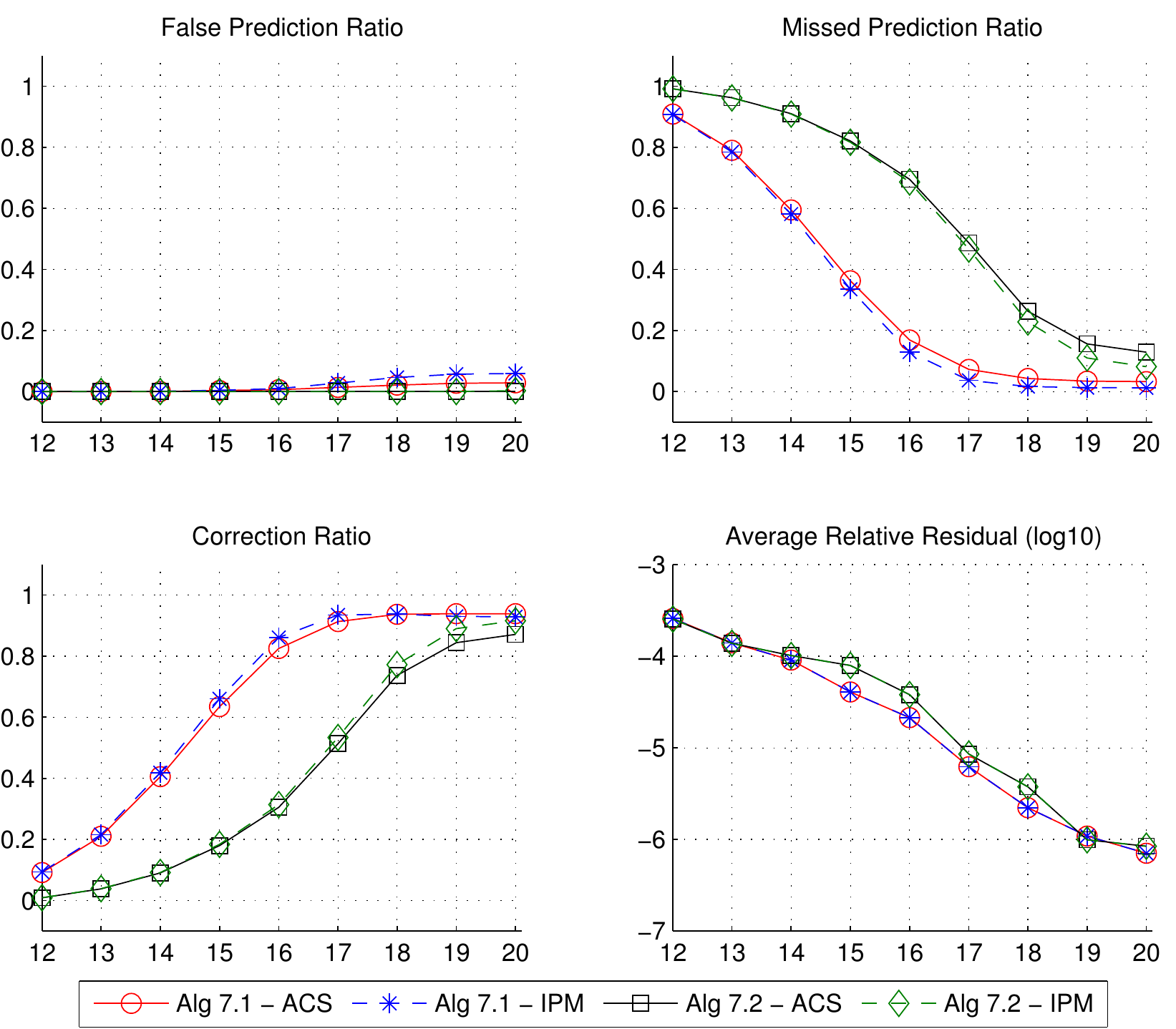} }
\caption{\small Prediction ratios for randomly generated {\qp} problems}
\label{fig:cr-nondegen_qp}
\end{minipage}
\hspace{0.4cm}
\begin{minipage}[t]{0.46\linewidth}
\centering\fbox{
\includegraphics[width=\textwidth]
{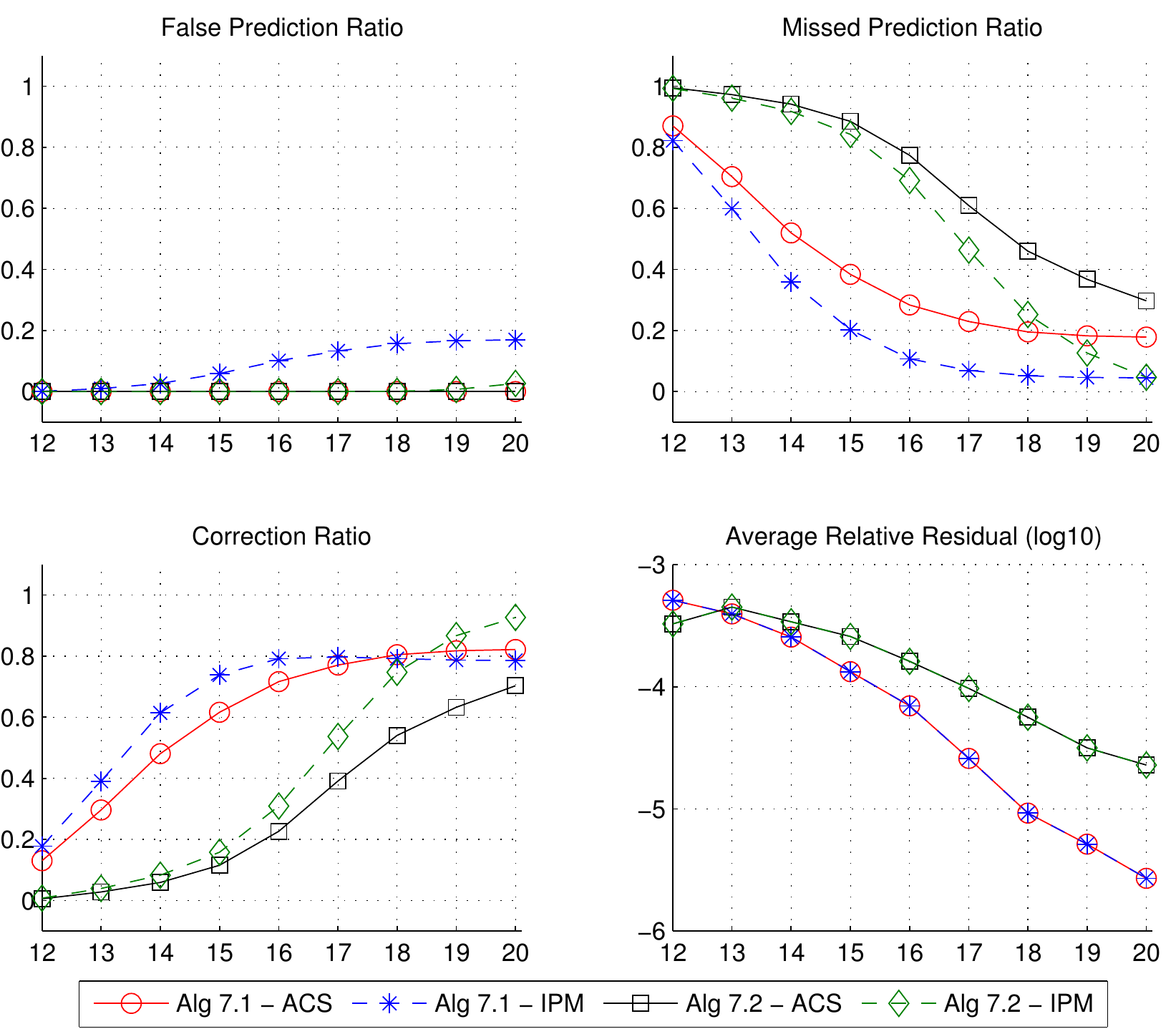} }
\caption{\small Prediction ratios for randomly generated primal-dual degenerate {\qp} problems}
\label{fig:cr-degen_qp}
\end{minipage}
\end{figure}

In Figures~\ref{fig:cr-nondegen_qp} and~\ref{fig:cr-degen_qp}, the x-axis gives the number of interior point iterations at which we terminate Algorithms~\ref{alg:perAlgQP} and~\ref{alg:unperAlgQP} and the y-aixs shows the average value of corresponding measures. The first three plots (from top to bottom, left to right) present the corresponding prediction ratios. In each plot, we compare the predicted active set from Algorithm~\ref{alg:perAlgQP} with that from the active-set solver (the red solid line with circle), Algorithm~\ref{alg:perAlgQP} with the interior point solver (the blue dashed line with star sign), Algorithm~\ref{alg:unperAlgQP} with the active-set solver (the black solid with square sign) and Algorithm~\ref{alg:unperAlgQP} with the interior point solver (green dashed line with diamond sign). The last figure shows the log10-scaled average relative residuals~\eqref{eqv:resResidualQP} of Algorithms~\ref{alg:perAlgQP} or~\ref{alg:unperAlgQP}.  

\begin{itemize}
\item Generally speaking, using perturbations yields earlier and better prediction of the original  optimal active set for both test cases, in terms of the correction ratios. 
Similar to the linear case, the correction ratios from the perturbed algorithms are over two times higher than that from the unperturbed ones  at some iterations, for test problems in both {\tspndegqp} and {\tspddegqp}.
\item The perturbed algorithm is more likely to predict the active set at an original optimal solution generated by the active-set solver. 
Although it is not obvious for  test problems in {\tspndegqp}, the difference is much clearer for the degenerate case {\tspddegqp}. In Figure~\ref{fig:cr-degen_qp}, the false-prediction ratio for Algorithm~\ref{alg:perAlgQP} and the interior point solver is about $17\%$ at the 20\textsuperscript{th} iteration but that for Algorithm~\ref{alg:perAlgQP} and the active-set solver stays close to 0. 
\item In Figure~\ref{fig:cr-degen_qp}, we can also observe that after the 18\textsuperscript{th} iteration, the average correction ratios comparing Algorithm~\ref{alg:unperAlgQP} with the {\ipm} solver are better than that comparing 
Algorithm~\ref{alg:perAlgQP} with active-set solver. This is because at the last few iterations the perturbations are not zero (on average about $\mathcal{O}(10^{-3})$)  and cannot shrink further;  so the iterates of Algorithm~\ref{alg:perAlgQP} cannot keep moving closer to the original optimal solution, which prevents Algorithm~\ref{alg:perAlgQP} from improving the correction ratios.
\item Ultimately, the correction ratio comparing Algorithm~\ref{alg:unperAlgQP} with the interior point solver should go to 1 but then it would need to solve the problems to high accuracy ($10^{-8}$). 
As our implementation is for proof of concept,  it can experience numerical issues when solving too far.
\item Another interesting phenomenon is that the relative residual of Algorithm~\ref{alg:perAlgQP} seems to decrease faster than that of Algorithm~\ref{alg:unperAlgQP}. It suggests that using perturbations may help stabilise the Newton system and thus generate better search directions, especially for the degenerate  problems in {\tspddegqp}.
\end{itemize}

\subsection{Solving the sub-problems}
\label{sec:solveSub_qp}
In this test, we first run Algorithm~\ref{alg:perAlgQP} and terminate it when $\mup^{k} < 10^{-3}$, record the number of interior point iterations, remove zero variables and corresponding columns and/or rows of $H$, $A$ and $c$ from the original problem~\eqref{eqv:orgQP}, and then solve the newly-formulated smaller-sized problem (sub-problem) using the active-set method. For comparison purposes we perform the same number of interior point iterations of Algorithm~\ref{alg:unperAlgQP}, predict the active set, formulate the sub-problem and solve it. We compare the number of active-set iterations used to solve the sub-problems from Algorithms~\ref{alg:perAlgQP} and~\ref{alg:unperAlgQP}. 

It is also essential to make sure the sub-problems that we generate are equivalent to their original problems. 
Assume $\Actvqp^{k}$ is the predicted active set when terminating the interior point process at iteration $k$, $x^{*}_{\rm sub}$ the optimal solution of the subproblems from the active-set solver and $x^{*}$ an optimal solution of the original problem. Let $ \Actvqp_{c}^{k} =  \{1,\ldots,n\} \setminus \Actvqp^{k}$ be the complement of $\Actvqp^{k}$.
We  consider the feasibility errors in the context of the original problem and the relative difference between the optimal objective values of the sub-problems and that of the original problems, namely,
\begin{equation}
\label{eqv:feaErrorTest}
\displaystyle \text{Feasibility error} = \frac{ \| A_{ \Actvqp_{c}^{k}} x^{*}_{\rm sub} - b\|_{\infty}}{1+\| b \|_{\infty}},
\end{equation}
and
\begin{equation}
\label{eqv:objErrorTest}
\displaystyle \text{Objective error} = \frac{ | c^{T}_{ \Actvqp_{c}^{k}} x^{*}_{\rm sub} + \frac{1}{2} (x^{*}_{\rm sub})^{T} H_{ \Actvqp_{c}^{k}} x^{*}_{\rm sub} - c^{T} x^{*}  -  \frac{1}{2} (x^*)^{T} H x^{*}\,|}{1+|\, c^{T} x^{*} + \frac{1}{2} (x^*)^{T} H x^{*}|},
\end{equation}
where $H_{ \Actvqp_{c}^{k}} = (H_{ij})_{i,j \in \Actvqp_{c}^{k}}$.
If the feasibility error is small, $\bar{x}^{*}$ with $\bar{x}^{*}_{\Actvqp^{k}} = 0$ and $\bar{x}^{*}_{\Actvqp_{c}^{k}} = x^{*}_{\rm sub}$ is a feasible point for the original {\qp}, and also optimal if the objective error is small as well.

\paragraph*{Randomly generated problems ({\tspndegqp} and {\tspddegqp})}
Table~\ref{tab-compare_actvset_iters} shows the average number of active-set iterations for the test problems in {\tspndegqp} and {\tspddegqp}. It is clear that using perturbations saves a lot of active-set iterations, about $63\%$ for problems in {\tspndegqp} and $36\%$ for {\tspddegqp}. Though unfortunately degeneracy seems to disadvantage the improvement, it cannot cover the fact that using perturbations would enhance the capabilities of predicting a better active set of the original problem, in the context of primal-dual path-following {\ipm}, and potentially reduce the computational effort for solving a problem.
\begin{table}[h]
\centering
\caption{Comparing the number of active-set iterations for Algorithms~\ref{alg:perAlgQP} and~\ref{alg:unperAlgQP}}
\label{tab-compare_actvset_iters}
{\small
\begin{tabular}{rcc|cc|}
 & \multicolumn{2}{c}{Random problems} & \multicolumn{2}{c}{Random degenerate problems}  \\
 \cline{2-5}
 \multicolumn{1}{r|}{}& Algorithm~\ref{alg:perAlgQP} & Algorithm~\ref{alg:unperAlgQP} & Algorithm~\ref{alg:perAlgQP} & Algorithm~\ref{alg:unperAlgQP} \\
\hline
\multicolumn{1}{r|}{Avg. \# of active-set iters}  &   46  &   143 &  190 &  300\\ 
\multicolumn{1}{r|}{Avg. $\mup^{k}$ and $\mu^{k}$ when terminate {\ipm}}  &    $5.8\times 10^{-04}$   &   $8.0\times 10^{-04}$ &  $6.3\times 10^{-04}$ &  $7.8 \times 10^{-04}$   \\
\hline
\hline
\end{tabular} }
\end{table}

We check the objective and feasibility errors in Table~\ref{tab-compare_actvset_rel_errors}.
All optimal solutions of the sub-problems generated from Algorithms~\ref{alg:perAlgQP} and~\ref{alg:unperAlgQP} are primal feasible for the original~\eqref{eqv:orgQP}.
For problems in {\tspndegqp}, Algorithm~\ref{alg:perAlgQP} yields small average objective error, in the order of $10^{-7}$.  For {\tspddegqp}, the average error from Algorithm~\ref{alg:unperAlgQP} is slightly higher, which is in the order of $10^{-6}$, but still acceptable, especially $90\%$  of the test problems in {\tspddegqp} have small relative errors, in the order of $10^{-16}$ (can be considered as zero in {\sc matlab}). This is, to some extend, even better than the result for the test case {\tspndegqp}.
\begin{table}[h]
\centering
\caption{Comparing the relative errors for Algorithms~\ref{alg:perAlgQP} and~\ref{alg:unperAlgQP}}
\label{tab-compare_actvset_rel_errors}
{\small
\begin{tabular}{rcc|cc|}
 & \multicolumn{2}{c}{Random problems} & \multicolumn{2}{c}{Random degenerate problems}  \\
 \cline{2-5}
 \multicolumn{1}{r|}{}& Algorithm~\ref{alg:perAlgQP} & Algorithm~\ref{alg:unperAlgQP} & Algorithm~\ref{alg:perAlgQP} & Algorithm~\ref{alg:unperAlgQP} \\
\hline
\multicolumn{1}{r|}{Avg. objective errors}  					& $2.0 \times 10^{-07}$ & $9.2\times 10^{-17}$& $6.4\times 10^{-06}$ &  $8.9 \times  10^{-17}$  \\ 
\multicolumn{1}{r|}{$90^{th}$ percentile of relative errors}	& $4.9 \times 10^{-07}$ & $3.3\times 10^{-16}$& $6.2\times 10^{-16}$ &	 $3.5 \times  10^{-16}$  \\
\multicolumn{1}{r|}{Avg. feasibility errors}  					& $ 5.4\times 10^{- 14}$ & $ 5.9\times 10^{-14 }$& $ 6.4 \times 10^{- 14}$ &  $ 8.2 \times  10^{- 14}$  \\ 
\hline\hline
\end{tabular} }
\end{table} 

\paragraph*{QP problems from the Netlib and Maros and Meszaros' test sets (\tsnqp)}
We also observe good numerical results for  a small set of {\qp} problems from Netlib and Maros and Meszaros'  convex {\qp} test set ({\tsnqp}). We summarise the results in Table~\ref{tab-solving_subprob+netlib_cute}.
For these problems, we save almost $50\%$ of active-set iterations and all optimal solutions of the sub-problems from Algorithm~\ref{alg:perAlgQP} are feasible and optimal for the original problems. For details, see Section~\ref{apd-netlib_cuter_sub}.
\begin{table}[h]
\caption{Numerical results for solving sub-problems for test case {\tsnqp}}
\label{tab-solving_subprob+netlib_cute}
\centering
{\small
\begin{tabular}{rcc}
 \multicolumn{1}{r}{}& Algorithm~\ref{alg:perAlgQP} & Algorithm~\ref{alg:unperAlgQP}  \\
\cline{2-3}
\multicolumn{1}{r|}{Avg. \# of active-set iters}  					& 6  & \multicolumn{1}{c|}{13}\\
\multicolumn{1}{r|}{Avg. $\mup^{k}$ and $\mu^{k}$ when terminate {\ipm}}  							&  $4.6\times 10^{-04}$  & \multicolumn{1}{c|}{$6.4\times 10^{-04}$     }\\
\multicolumn{1}{r|}{Avg. relative errors}  					& $1.1 \times 10^{-15}$ & \multicolumn{1}{c|}{$1.8\times 10^{-15} $}\\
\multicolumn{1}{r|}{$90^{th}$ percentile of relative errors}	                 & $9.2 \times 10^{-16}$ & \multicolumn{1}{c|}{$9.9 \times 10^{-16} $}\\
\multicolumn{1}{r|}{Avg. feasibility errors}  					& $  9.6 \times 10^{-13 }$ & \multicolumn{1}{c|}{$ 8.8 \times 10^{- 13} $}\\
\cline{2-3}
\end{tabular} }
\end{table} 

\section{Conclusions}
Theoretically, we have extended the idea of active-set prediction using controlled perturbations from {\lp} to {\qp}. 
Numerically, we have obtained satisfactory  preliminary results. 
Based on our observations, it seems that for the purpose of  optimal active-set prediction for {\ipms} for {\qp} problems, and the idea of using controlled perturbations is promising.

Note that our implementation of Algorithm~\ref{alg:perAlgQP} is preliminary. 
We have not employed techniques such as the predictor-corrector or multiple centralities~\cite{Gondzio25year}. Thus the algorithm may not be efficient enough and needs further refinement.

\paragraph{Acknowledgement}
I am grateful to Dr. Coralia Cartis for useful discussions.

\section*{References}
\bibliographystyle{elsarticle-num} 
\bibliography{bibliography}

\begin{thebibliography}{10}
\expandafter\ifx\csname url\endcsname\relax
  \def\url#1{\texttt{#1}}\fi
\expandafter\ifx\csname urlprefix\endcsname\relax\def\urlprefix{URL }\fi
\expandafter\ifx\csname href\endcsname\relax
  \def\href#1#2{#2} \def\path#1{#1}\fi

\bibitem{CartisYan2014}
C.~Cartis, Y.~Yan, Active-set prediction for interior point methods using
  controlled perturbations, Technical Report ERGO-14-006, School of
  Mathematics, The University of Edinburgh (April 2014).

\bibitem{bakry}
A.~S. El-Bakry, R.~Tapia, Y.~Zhang, A study of indicators for identifying zero
  variables in interior point methods, SIAM Review 36~(1) (1994) 45--72.

\bibitem{Berkelaar1996}
A.~Berkelaar, B.~Jansen, K.~Roos, T.~Terlaky, Basis- and tripartition
  identification for quadratic programming and lineair complementary problem s;
  from an interior solution to an optimal basis and viceversa, Econometric
  Institute Report EI 9614-/A, Erasmus University Rotterdam, Econometric
  Institute (Jan. 1996).

\bibitem{Berkelaar1997}
A.~Berkelaar, K.~Roos, T.~Terlaky,
  \href{http://dx.doi.org/10.1007/978-1-4615-6103-3_6}{The optimal set and
  optimal partition approach to linear and quadratic programming}, in: T.~Gal,
  H.~Greenberg (Eds.), Advances in Sensitivity Analysis and Parametic
  Programming, Vol.~6 of International Series in Operations Research \&
  Management Science, Springer US, 1997, pp. 159--202.
\newblock \href {http://dx.doi.org/10.1007/978-1-4615-6103-3_6}
  {\path{doi:10.1007/978-1-4615-6103-3_6}}.
\newline\urlprefix\url{http://dx.doi.org/10.1007/978-1-4615-6103-3_6}

\bibitem{goldman1956polyhedral}
A.~J. Goldman, A.~W. Tucker, Polyhedral convex cones, Linear inequalities and
  related systems 38 (1956) 19--40.

\bibitem{wright}
S.~J. Wright, Primal-Dual Interior-Point Methods, SIAM, 1997.

\bibitem{Bertsekas}
D.~Bertsekas, Nonlinear Programming, Athena Scientific, 1999.

\bibitem{Nocedal2006}
J.~Nocedal, S.~Wright, Numerical Optimization, Springer, 2006.

\bibitem{Boyd}
S.~Boyd, L.~Vandenberghe, Convex Optimization, Cambridge University Press, New
  York, NY, USA, 2004.

\bibitem{Monteiro1989QP}
R.~D. Monteiro, I.~Adler, \href{http://dx.doi.org/10.1007/BF01587076}{Interior
  path following primal-dual algorithms. {P}art {II}: Convex quadratic
  programming}, Mathematical Programming 44~(1-3) (1989) 43--66.
\newblock \href {http://dx.doi.org/10.1007/BF01587076}
  {\path{doi:10.1007/BF01587076}}.
\newline\urlprefix\url{http://dx.doi.org/10.1007/BF01587076}

\bibitem{guler1993convergence}
O.~G{\"u}ler, Y.~Ye, Convergence behavior of interior-point algorithms,
  Mathematical Programming 60~(1-3) (1993) 215--228.

\bibitem{ye2011interior}
Y.~Ye, Interior point algorithms: theory and analysis, Vol.~44, John Wiley \&
  Sons, 2011.

\bibitem{Stephen2009}
S.~L. Campbell, C.~D. Meyer,
  \href{http://epubs.siam.org/doi/abs/10.1137/1.9780898719048}{Generalized
  Inverses of Linear Transformations}, Society for Industrial and Applied
  Mathematics, 2009.
\newblock \href
  {http://arxiv.org/abs/http://epubs.siam.org/doi/pdf/10.1137/1.9780898719048}
  {\path{arXiv:http://epubs.siam.org/doi/pdf/10.1137/1.9780898719048}}, \href
  {http://dx.doi.org/10.1137/1.9780898719048}
  {\path{doi:10.1137/1.9780898719048}}.
\newline\urlprefix\url{http://epubs.siam.org/doi/abs/10.1137/1.9780898719048}

\bibitem{Gondzio25year}
J.~Gondzio, {Interior point methods 25 years later.}, Eur. J. Oper. Res.
  218~(3) (2012) 587--601.
\newblock \href {http://dx.doi.org/10.1016/j.ejor.2011.09.017}
  {\path{doi:10.1016/j.ejor.2011.09.017}}.

\bibitem{mehrotra}
S.~Mehrotra, On the implementation of a primal-dual interior point method, SIAM
  Journal on Optimization 2~(4) (1992) 575--601.
\newblock \href {http://dx.doi.org/http://dx.doi.org/10.1137/0802028}
  {\path{doi:http://dx.doi.org/10.1137/0802028}}.

\bibitem{netlib}
D.~M. Gay, Electronic mail distribution of linear programming test problems,
  Mathematical Programming Society COAL Newsletter 13 (1985) 10--12.

\bibitem{maros1999repository}
I.~Maros, C.~M{\'e}sz{\'a}ros, A repository of convex quadratic programming
  problems, Optimization Methods and Software 11~(1-4) (1999) 671--681.

\bibitem{Mangasarian}
O.~L. Mangasarian, J.~Ren, New improved error bounds for the linear
  complementarity problem, Mathematical Programming 66~(2) (1994) 241--255.

\end{thebibliography}

\appendix

\section{Proof of Lemma~\ref{lem-eb_qp}}
\label{apd-proof_lemma_eb_qp}
\renewcommand{\theequation}{A.\arabic{equation}}
\setcounter{theorem}{0}
\renewcommand{\thetheorem}{A.\arabic{theorem}}
We follow the approach in~\cite[Appendix A]{CartisYan2014} but apply it to~\eqref{eqv:orgQP} problems.
Substituting $s = c + Hx -A^{\top}y$ and $y=y^{+}-y^{-}$, where $y^{+}=\max(y,0)$ and $y^{-}=-\min(y,0)$ into the first order optimality conditions~\eqref{eqv:kkt_perturbedQP} with $\lambda = 0$, we can verify that finding an optimal solution of~\eqref{eqv:orgQP} is equivalent to solving the following {\lcp} problem,
\begin{equation}
\label{eqv:lcp}
\quad Mz+q \geq 0, \quad z \geq 0, \quad z^{T}(Mz+q)=0,\quad
\end{equation}
where $Q$, $A$, $b$ and $c$ are~\eqref{eqv:orgQP} problem data, $\pntp \in \Real^{n} \times \Real^{m} \times \Real^{n}$ and $z$ is considered to be the vector of variables, and
where 
\begin{equation}
\label{eqv:LCP_Mzq_qp}
M = 
\begin{bmatrix}
H     & 	-A^{T}   & A^{T} 	\\
A     & 	0	    & 	     0 \\
-A   &	0 	    &       0  	
\end{bmatrix}
,\quad
q = 
\begin{bmatrix}
c \\
-b \\
b
\end{bmatrix}
\quad \mbox{and}\quad
z = 
\begin{bmatrix}
x \\
y^{+} \\
y^{-}
\end{bmatrix}.
\end{equation}

\begin{lemma}
\label{lem:M_psd}
The matrix $M$, defined in~\eqref{eqv:LCP_Mzq_qp}, is positive semidefinite, and so~\eqref{eqv:lcp} is a monotone {\lcp}.
\end{lemma}
\begin{proof}
$\forall v \neq 0, v=(v_{1},v_{2},v_{3})$, where $v_{1} \in \Real^{n}, v_{2} \in \Real^{m}$ and $v_{3} \in \Real^{m}$.
$v^{T} M v= v^{T}_{1}Hv_{1} + v^{T}_{2}Av_{1} - v^{T}_{3}Av_{1} - v_{1}^{T}A^{T}v_{2} + v^{T}_{1}A^{T}v_{3}$. Since $v^{T}_{2}Av_{1} = (v^{T}_{2}Av_{1})^{T}=v^{T}_{1}A^{T}v_{2}$ and $v^{T}_{3}Av_{1}=(v^{T}_{3}Av_{1})^{T}=v^{T}_{1}A^{T}v_{3}$, we have $v^{T} M v = v^{T}_{1}Hv_{1} \geq 0$ as $H$ is positive semi-definite. Thus
$M$ is positive semi-definite.
\end{proof}

A global error bound for a monotone {\lcp}~\cite{Mangasarian} has already been present in~\cite[Appendix A]{CartisYan2014}. We restate it here for clarity.
 
\begin{lemma}[{Mangasarian and Ren~\cite[Corollary 2.2]{Mangasarian}}]
\label{lem-eb}
Let $z$ be any point away from the solution set of a monotone {\lcp}(M,q)~\eqref{eqv:lcp} and $z^{*}$ be the closest solution of~\eqref{eqv:lcp} 
to $z$ under the norm $\|\cdot\|$.
Then $r(z)+w(z)$ is a global error bound for~\eqref{eqv:lcp}, namely,
\[ 
	\| z-z^{*}\| \leq \tau (r(z)+w(z)), 
\]
where $\tau$ is some problem-dependent constant, independent of $z$ and $z^{*}$, and
\begin{equation}
\label{eqv:r_w}
\quad r(z)=\| z-(z-Mz-q)_{+}\|
\quad
\text{and}
\quad 
w(z)=\| \left( -Mz-q,-z,z^{T}(Mz+q) \right)_{+}\|. \quad
\end{equation}
\end{lemma}

In~\cite[Theorem A.5]{CartisYan2014}, we present an error bound for {\lp}. It is straightforward to extend this result to {\qp} problems. We state the following lemma without giving a proof.
\begin{lemma}[Error bound for {\qp}]
\label{lem_errorbounds_qp}
Let $\pntp \in \Real^{n}\times\Real^{m}\times\Real^{n}$ where $s = c - A^{T}y + Hx  $. Then there exist a solution $(x^{*},y^{*},s^{*})$ of~\eqref{eqv:orgQP} and problem-dependent constants $\tau_{p}$ and $\tau_{d}$, independent of $\pntp$ and $\optSol$, such that
\begin{equation*}
\|x-x^{*}\| \leq \tau_{p}\left(r(x,y) + w(x,y) \right)
\quad
\text{and}
\quad
\|s-s^{*}\| \leq \tau_{d} \left(r(x,y) + w(x,y) \right),
\end{equation*}
where
\begin{equation}
\label{r_lp}
    r(x,y)  = \left\| \left(\cmin{x}{s},\,\, \cmin{y^{+}}{Ax-b},\,\, \cmin{y^{-}}{-Ax + b}\right) \right\|,
\end{equation}
and
\begin{equation}
\label{w_lp}
w(x,y) = \| (-s,\,\,b-Ax,\,\,Ax-b,\,\,-x,\,\,c^{T}x - b^{T}y + x^{T}Hx)_{+}\|,
\end{equation}
and where $\cmin{x}{s} = \left(\, \min(x_{i},s_{i}) \,\right)_{ i = 1,\ldots, n }$, $y^{+} = \cmax{y}{0}$ and $y^{-} = -\cmin{y}{0}$.
\end{lemma}

\paragraph{Proof of Lemma~\ref{lem-eb_qp}} Considering $Ax = b$ and $A^{T}y+s-Hx = c$, this result follows from Lemma~\ref{lem_errorbounds_qp}. 

\section{An active-set prediction procedure}
\label{appnd:actvprocedure}
\setcounter{algorithm}{0}
\renewcommand{\thealgorithm}{B.\arabic{algorithm}}
In our numerical test, we apply the following strategy to predict the active constraints.
We partition the index set $\{ 1,2,\ldots,n  \}$ into three sets, $\Actv^{k}$ as the predicted active set, $\Iactv^{k}$ as the predicted inactive set and $\Nactv^{k} = \{ 1,2,\ldots,n  \} \backslash \left( {\Actv}^{k} \cup {\Iactv}^{k}\right)$ which includes all undetermined indices.
During the running of the algorithm, we move indices between these sets according to the threshold tests 
$x^{k}_{i} < C$ and $s^{k}_{i} > C,$
where $C$ is a user-defined threshold; $C = 10^{-5}$ in our tests.
Initialise $\Actv^{0} = \Iactv^{0} = \emptyset$ and $\Nactv^{0} = \{  1,2,\dots,n \}$.
An index is moved from $\Nactv^{k}$ to $\Actv^{k}$ if the threshold test is satisfied for two consecutive  iterations, otherwise from $\Nactv^{k}$ to $\Iactv^{k}$.
We move an index from $\Actv^{k}$ to $\Nactv^{k}$ if the threshold test is not satisfied at the current iteration.
An index is moved from $\Iactv^{k}$ to $\Nactv^{k}$ if the threshold test  is satisfied at the current iteration. We summarise the above as Procedure~\ref{alg:actvPrediction}.
\begin{algorithm}[htb]         
\floatname{algorithm}{Procedure}            
\caption{An Active-set Prediction Procedure}          
\label{alg:actvPrediction}                           
{\footnotesize
\begin{algorithmic}
	\STATE{\textbf{Initialise}: $A^{0} = \Iactv^{0} = \emptyset$ and $\Nactv^{0} = \{  1,2,\dots,n \}$.}
	\STATE{\textbf{At $k^{th}$ iteration, $k > 1$},}
	\FOR{$i = 1,\dots,n$}
		\IF{$i \in \Nactv^{k}$}
			\IF{the threshold test is satisfied for iterations $k-1$ and $k$}
				\STATE{$\Actv^{k} = \Actv^{k} \cup \{i\}$ and $\Nactv^{k} = \Nactv^{k}\backslash\{i\}$;}
			\ELSE
				\STATE{$\Iactv^{k} = \Iactv^{k} \cup \{i\}$ and $\Nactv^{k} = \Nactv^{k}\backslash\{i\}$.}
			\ENDIF
			\IF{$i \in \Actv^{k}$ and the threshold test is not satisfied}
				\STATE{$\Actv^{k} = \Actv^{k} \backslash \{i\}$ and $\Nactv^{k} = \Nactv^{k}\cup\{i\}$;}
			\ENDIF
			\IF{$i \in \Iactv^{k}$ and the threshold test is satisfied}
				\STATE{$\Iactv^{k} = \Iactv^{k} \backslash \{i\}$ and $\Nactv^{k} = \Nactv^{k}\cup\{i\}$.}
			\ENDIF
		\ENDIF
	\ENDFOR
\end{algorithmic} }
\end{algorithm}

\begin{landscape}
\section{Numerical results for solving sub-problems}
\label{apd-netlib_cuter_sub}
In Table~\ref{tab-sumSolveSub}, from the left to the right, we present the name the problem, the number of equality constraints and variables, 	the value of duality gap when terminate Algorithms~\ref{alg:perAlgQP} and~\ref{alg:unperAlgQP}, the number of active-set iterations for solving the subproblems generated from Algorithms~\ref{alg:perAlgQP} and~\ref{alg:unperAlgQP}, the primal feasibility errors for the optimal solutions of the subproblems from Algorithms~\ref{alg:perAlgQP} and~\ref{alg:unperAlgQP},  and the objective errors between the subproblem and the original problem. 
\begin{footnotesize}
\begin{center}
\begin{longtable}{| l | l | l | l | l | l | l | l | l | l | l | l |}
\caption[Solving sub-problem test on a selection of Netlib and Maros and Meszaros' convex {\qp} problems]{%
\normalsize Solving sub-problem test on a selection of Netlib and Maros and Meszaros' convex {\qp} problems.} 
\label{tab-sumSolveSub} \\

\hline \multicolumn{1}{|c|}{\textbf{Probs}} %
    & \multicolumn{1}{c|}{\textbf{m}} %
    & \multicolumn{1}{c|}{\textbf{n}} %
    & \multicolumn{1}{c|}{\textbf{$\mu^{K}_{\lambda}$}} %
    & \multicolumn{1}{c|}{\textbf{$\mu^{K}$}} %
    & \multicolumn{1}{c|}{\textbf{IPM Itr}} %
    & \multicolumn{1}{c|}{\textbf{actvItr Per}} %
    & \multicolumn{1}{c|}{\textbf{actvItr Unp}} 
    & \multicolumn{1}{c|}{\textbf{feaErr Per}} %
    & \multicolumn{1}{c|}{\textbf{feaErr Unp}} %
    & \multicolumn{1}{c|}{\textbf{relObjErr Per}} %
    & \multicolumn{1}{c|}{\textbf{relObjErr Unp}} \\ %
\hline 
\endfirsthead

\multicolumn{12}{c}%
{{\bfseries \tablename\ \thetable{} -- continued from previous page}} \\
\hline \multicolumn{1}{|c|}{\textbf{Probs}} %
    & \multicolumn{1}{c|}{\textbf{m}} %
    & \multicolumn{1}{c|}{\textbf{n}} %
    & \multicolumn{1}{c|}{\textbf{$\mu^{K}_{\lambda}$}} %
    & \multicolumn{1}{c|}{\textbf{$\mu^{K}$}} %
    & \multicolumn{1}{c|}{\textbf{IPM Itr}} %
    & \multicolumn{1}{c|}{\textbf{actvItr Per}} %
    & \multicolumn{1}{c|}{\textbf{actvItr Unp}} 
    & \multicolumn{1}{c|}{\textbf{feaErr Per}} %
    & \multicolumn{1}{c|}{\textbf{feaErr Unp}} %
    & \multicolumn{1}{c|}{\textbf{relObjErr Per}} %
    & \multicolumn{1}{c|}{\textbf{relObjErr Unp}} \\ %
\hline 
\endhead

\hline \multicolumn{12}{|r|}{{Continued on next page}} \\ \hline
\endfoot

\hline \hline
\endlastfoot

QP\_ADLITTLE &   55 &  137 &      7.9e-04 &      9.6e-04 &           13 &            3 &           22 &      1.5e-12 &      1.0e-12 &      0.0e+00 &      1.6e-16 \\
  QP\_AFIRO &   27 &   51 &      1.9e-04 &      2.7e-04 &           13 &            1 &            5 &      2.9e-13 &      3.3e-13 &      7.2e-16 &      4.3e-16 \\
  QP\_BLEND &   74 &  114 &      2.9e-04 &      3.2e-04 &           14 &            7 &           38 &      5.4e-13 &      4.8e-13 &      4.8e-16 &      9.9e-15 \\
  QP\_SC50A &   49 &   77 &      9.2e-05 &      1.5e-04 &           10 &            1 &            1 &      2.6e-13 &      2.6e-13 &      7.5e-16 &      7.5e-16 \\
  QP\_SC50B &   48 &   76 &      5.3e-04 &      7.9e-04 &            8 &            2 &            3 &      3.2e-13 &      3.9e-13 &      1.2e-16 &      5.9e-16 \\
 QP\_SCAGR7 &  129 &  185 &      8.6e-04 &      1.3e-03 &           15 &            1 &           10 &      1.0e-11 &      9.5e-12 &      2.4e-16 &      2.4e-16 \\
QP\_SHARE2B &   96 &  162 &      1.2e-04 &      1.4e-04 &           20 &            4 &           12 &      6.0e-12 &      4.7e-12 &      1.6e-14 &      2.3e-14 \\
  CVXQP1\_S &  150 &  200 &      4.5e-04 &      7.8e-04 &            8 &            1 &           16 &      6.0e-14 &      6.1e-14 &      1.6e-16 &      1.6e-16 \\
  CVXQP2\_S &  125 &  200 &      6.5e-04 &      1.1e-03 &            8 &            1 &           48 &      3.5e-14 &      4.0e-14 &      9.2e-16 &      4.6e-16 \\
  CVXQP3\_S &  175 &  200 &      5.4e-04 &      6.6e-04 &            9 &            2 &            4 &      6.3e-14 &      5.4e-14 &      4.7e-16 &      4.7e-16 \\
     DUAL1 &   86 &  170 &      5.2e-04 &      6.1e-04 &            2 &           29 &           29 &      6.5e-15 &      6.5e-15 &      0.0e+00 &      0.0e+00 \\
     DUAL2 &   97 &  192 &      5.1e-04 &      6.4e-04 &            2 &            5 &            5 &      6.4e-15 &      6.4e-15 &      0.0e+00 &      0.0e+00 \\
     DUAL3 &  112 &  222 &      5.9e-04 &      6.1e-04 &            3 &           15 &           15 &      1.3e-14 &      1.3e-14 &      0.0e+00 &      0.0e+00 \\
     DUAL4 &   76 &  150 &      3.0e-04 &      4.3e-04 &            4 &           14 &           14 &      1.3e-14 &      1.3e-14 &      0.0e+00 &      0.0e+00 \\
     HS118 &   44 &   59 &      1.8e-04 &      2.8e-04 &            8 &            0 &           15 &      2.3e-14 &      1.5e-13 &      3.2e-16 &      0.0e+00 \\
      HS21 &    3 &    5 &      3.4e-04 &      6.6e-04 &           10 &            2 &            2 &      5.2e-14 &      5.2e-14 &      0.0e+00 &      0.0e+00 \\
      HS51 &    3 &   10 &      9.2e-04 &      7.3e-04 &            3 &           20 &           20 &      3.1e-15 &      3.1e-15 &      0.0e+00 &      0.0e+00 \\
      HS53 &    8 &   10 &      9.9e-04 &      1.9e-03 &            6 &            1 &            1 &      2.2e-14 &      2.2e-14 &      0.0e+00 &      0.0e+00 \\
      HS76 &    3 &    7 &      7.9e-05 &      1.5e-04 &            6 &            1 &            3 &      8.9e-16 &      1.9e-15 &      1.6e-16 &      0.0e+00 \\
  ZECEVIC2 &    4 &    6 &      2.6e-04 &      4.0e-04 &            5 &            1 &            2 &      4.4e-15 &      3.9e-15 &      8.7e-16 &      0.0e+00 \\
\hline
  Average: &      &      &      4.6e-04 &      6.4e-04 &            8 &            6 &           13 &      9.6e-13 &      8.6e-13 &      1.1e-15 &      1.8e-15 \\
90th Pctl: &      &      &       	     &                  &               &              &              &      6.0e-12 &      4.7e-12 &      9.2e-16 &      9.9e-15
\end{longtable}
\end{center}
\end{footnotesize}
\end{landscape}

\end{document}